\documentclass[english,10pt]{smfart}
\RequirePackage[T1]{fontenc}
\RequirePackage{amsfonts,latexsym,amssymb}

\RequirePackage{mathrsfs}
\let\mathcal\mathscr
 \usepackage{tikz-cd}
\usepackage{extarrows}
\usepackage{enumerate}
\usepackage{float}
\usepackage[all]{xy}
\usepackage{url}
\usepackage{fge}
\usepackage{color}
\usepackage{amscd}
\usepackage{makecell}

\numberwithin{equation}{section}
\theoremstyle{plain}
\newtheorem{theorem}{Theorem}[section]

\newtheorem{proposition}[theorem]{Proposition}
\theoremstyle{definition}
\newtheorem{definition}[theorem]{Definition}
\newtheorem{example}[theorem]{Example}

\theoremstyle{remark}
\newtheorem{remark}[theorem]{Remark}

            \usepackage{hyperref}

\let\cal\mathcal

\let\bb\mathbb

 \def\Z{{\bf Z}}
\def\C{{\bb C}}

\def\epsilon{\varepsilon}

\usepackage{microtype}
\allowdisplaybreaks

\makeindex

\begin{document}

\title[Parameters and theta lifts]{Parameters and theta lifts}

\author{Zhe Li}
\address{School of mathematical Sciences, Xiamen University, 422 South Siming Road, Siming District,
Xiamen, Fujian Province, China}
\email{zheli@xmu.edu.cn}
   
\author{Shanwen Wang}
\address{School of Mathematics, Renmin University of China, No. 59 Zhongguancun Street, Haidian District, 100872, Beijing, China}
\email{s\_wang@ruc.edu.cn}

\author{Zhiqi Zhu}
\address{Dipartimento di Matematica "Tullio Levi-Civita", Università di Padova, Via Trieste, 63, 35121 Padova, Italy}
\email{zhiqi.zhu@studenti.unipd.it}
\thanks{The research of the corresponding author (S.W.) was supported by the Fundamental Research Funds for the Central Universities, and the research Funds of Renmin University of China n\textsuperscript{o} 20XNLG04 and the National Natural Science Foundation of China (Grant n\textsuperscript{o} 11971035).}
\begin{abstract}
In this note, we make explicit the correspondence between Harish-Chandra parameters and Langlands-Vogan parameters for symplectic groups and orthogonal groups of equal rank over reals. As an application, we reformulate Moeglin's results \cite{Moe89} and Paul's work \cite{Paul05} on the Howe correspondence for symplectic-orthogonal dual pairs using Langlands-Vogan parameters.
\end{abstract}
\thanks{}
\setcounter{tocdepth}{3}
\maketitle
\let\languagename\relax

\stepcounter{tocdepth}
{\Small
\tableofcontents
}
\section{Introduction}
Throughout this paper, we fix an additive character $\psi:{\bb R}\rightarrow {\bb C}^{\times}$, set $\Gamma=\{1, \sigma\}$ the Galois group of ${\bb C}/{\bb R}$ and 
by a representation of a real reductive group $G$, we mean a smooth Fr\'echet representation of moderate growth. 

Given a real reductive group $G$, let $\mathfrak{g}$ be the Lie algebra of $G$ and we fix a maximal compact subgroup $K$ of $G$. Wallach \cite[Chapter~11]{Wal92} introduced an equivalence between the category of smooth Fr\'echet representations of moderate growth of $G$ and the category of admissible finitely generated $(\mathfrak{g}, K)$-modules. Since any such $(\mathfrak{g}, K)$-module can be parametrized by Harish-Chandra parameters, so are the representations of $G$.

Let $(V,(\cdot,\cdot))$ be a $2n$-dimensional real vector space equipped with a non-degenerate symplectic bilinear form and $(V^{'},\langle\cdot, \cdot\rangle)$ be a $(2n+2)$-dimensional real vector space equipped with a symmetric bilinear form. Let $G = {\rm Sp}(V)$ and $G' = {\rm O}(V^{'})$ be the isometry groups of $V$ and $V^{'}$ respectively. Moreover, let ${\bb V}= V\otimes V^{'}$ equipped with the symplectic form $(\cdot,\cdot)\otimes \langle\cdot, \cdot\rangle$ and we denote by $\widehat{\rm Sp}({\bb V})$ the metaplectic group associated to ${\bb V}$. The pair $(G, G')$ is a reductive dual pair in the symplectic group ${\rm Sp}(\bb V)$ (\cite[\S II.1]{Ku}), and there is a natural map
\[\iota: {\rm Sp}(V)\times {\rm O}(V^{'})\rightarrow {\rm Sp}(\bb V),\]
which can be lifted to a homomorphism\footnote{Such a splitting of $\widehat{\mathrm{Sp}}(\bb V)\rightarrow \mathrm{Sp}(\bb V)$ is not unique, which depends on some auxiliary data described in \cite[\S 3.2]{GI14}.}:
\begin{equation}\label{splitA}\iota_{V, V^{'}}: {\rm Sp}(V)\times {\rm O}(V^{'})\rightarrow \widehat{{\rm Sp}}(\bb V).\end{equation}

For the dual pair $(G,G')$, Roger Howe \cite{Howe} introduced the theta correspondence between the representations of $G$ and the representations of $G'$. His construction used the Weil representation $\omega_{\psi}$ of $\widehat{\mathrm{Sp}}({\bb V})$(cf. \cite[\S I.1]{Ku}), which is an extension of the oscillating representation of the Heisenberg group $H({\bb V})={\bb V}\oplus \mathbb{R}$, depending on the choice of $\psi$. If $\pi$ and $\pi^{'}$ are irreducible admissible representations of $G$ and $G'$ respectively, we say $\pi$ and $\pi^{'}$ \textit{correspond} if $\pi \otimes \pi^{'}$ is a quotient of the Weil representation $\omega_\psi$ of $\widehat{{\rm Sp}}(\bb V)$, restricted to $G\times G'$, denoted by $\theta_{V^{'},V}(\pi^{'})=\pi$ and $\theta_{V,V^{'}}(\pi)=\pi^{'}$. If $V'$ has signature $(p,q)$, we will replace the notation $\theta_{V,V'}$ by $\theta_{p,q}$.

Assume $(p,q)$ is the signature of the symmetric bilinear form $\langle\cdot, \cdot\rangle$. Moeglin \cite{Moe89} computed a significant part of the full correspondence where $p$ and $q$ are both even. Later on, using the Harish-Chandra parameter, Paul \cite[Theorem 15]{Paul05} gave a complete and explicit description of the correspondence by filling in the case\footnote{In this paper, we consider the equal rank case, which implies that $p,q$ are both even. }with $p$ and $q$ odd. Their description of the explicit theta correspondence for the symplectic-orthogonal dual pairs uses Harish-Chandra parameters. 

On the other hand, we can parametrize the representations of real reductive groups via Langlands-Vogan parameters, which is a more direct parametrization compared to Harish-Chandra parametrization. The goal of this paper is to translate Moeglin's results and Paul's results for the limit of discrete series representations via Langlands-Vogan parameters of the smooth Fréchet representations of moderate growth, which will serve as an input for our forthcoming paper on GGP conjecture for Fourier-Jacobi case over reals. As a consequence, we also formulate such translation for tempered representations.

We begin by giving a precise statement of the results (cf. Thm. \ref{main1} and Thm. \ref{main2}).
\subsection{Langlands-Vogan parameters}\label{LangVoganInd}

Let $H=G$ or $G'$. We fix a pinning of the complex dual group $H^{d}$
\[{\bf Spl}_{H^d}=({\cal B}, \cal{T}_{*}, \{{\cal X}_\alpha\}),\]
where ${\cal B}$ is a Borel subgroup of $H^d$,  $\cal{T}_{*}$ is a torus of ${\cal B}$, $\{{\cal X}_\alpha\}$ is the set of root vectors for the simple roots of $\cal{T}_{*}$ in $\cal{B}$(cf. \S \ref{pinningdef}). Let $W_{\bb R}$ be the Weil group of ${\bb R}$. 
The Langlands dual group $^LH$ of $H$ associated with the pinning ${\bf Spl}_{H^d}$ is the semi-direct product $H^d\rtimes W_{\bb R}$, where the action of $W_{\bb R}$ on $H^d$ factors through the projection $p_{W_{\bb R}}: W_{\bb R}\rightarrow {\rm Gal}({\bb C}/{\bb R})$ stabilizing ${\bf Spl}_{H^d}$. We remark that $^LH$ only depends on its inner class.  A \textbf{Langlands parameter} (or \textbf{L-parameter}) of $H$ is a continuous morphism $\varphi: W_{\bb R}\rightarrow {^LH}$ satisfying the following two conditions:
\begin{enumerate}
\item $p_{\bb R}\circ \varphi=\mathrm{Id}_{W_{\bb R}}$, where $p_{\bb R}: {^L H}\rightarrow W_{\bb R}$;
\item the image of the restriction map $\varphi_{|_{\mathbb{C}^{\times}}}: \mathbb{C}^{\times}\rightarrow H^d\times \mathbb{C}^{\times}$ consists of the elements which is semi-simple in $H^d$.
\end{enumerate} 
In particular, a Langlands parameter of $H$ is called tempered if its image is bounded.

Recall that the irreducible representations of $W_{\bb R}$ have dimension either $1$ or $2$ as $W_{\bb R}$ has an abelian subgroup of index $2$. 

\begin{enumerate}

\item[(1)] The $1$-dimensional representations of $W_{\bb R}$ are the quasi-characters of $W_{\bb R}^{\rm ab}={\bb R}^{\times}$, which are of the form $\chi_{\epsilon, s}(x)={\bf sgn}(x)^{\frac{\epsilon-1}{2}} \vert x\vert^s$, for $\epsilon\in\{\pm1\}$ and $s\in {\bb C}$.
 
 \item[(2)] 
 The irreducible representations of $W_{\bb R}$ of dimension $2$ are of the form  \[\rho_{m,s}={\rm Ind}_{{\bb C}^{\times}}^{W_{\bb R}}\chi_{m,s},\] for $\chi_{m,s}(z)=z^m\cdot(z\bar{z})^{-m/2}\cdot \vert z\vert_{\bb C}^s$ a quasi-character of ${\bb C}^{\times}$ with $m\in {\bb Z}$ and $s\in {\bb C}$. We will denote $\rho_{m,0}$ by $\rho_m$.

\end{enumerate}
Thus, a Langlands parameter $\varphi$ of $H$ can be represented using its decomposition into direct sum of irreducible representations of $W_{\bb R}$.

The complex dual group $H^d$ of $H$ acts by conjugation on the set of Langlands parameters of $H$. We denote by $\Phi(H)$ the set of $H^d$-conjugacy classes of the Langlands parameters of $H$ and by $\Phi_{{\rm temp}}(H)$ the set of conjugacy classes of the tempered Langlands parameters. The classification theorem of Langlands says that there exists a partition of the set $\Pi(H)$ of the equivalent classes of irreducible representations of $H$:
\[\Pi(H)=\coprod_{\varphi\in \Phi(H)}\Pi(\varphi, H),\]
where $\Pi(\varphi,H)$'s are finite sets of irreducible representations of $H$, called the {\bf L-packets} (or {\bf Langlands packets}).
This partition satisfies a number of properties and among them, the most important ones for us are the followings:
\begin{enumerate}
\item all the elements of a L-packet admit the same infinitesimal character;
\item for the set $\Pi_{{\rm temp}}(H)$ of all tempered representations of $H$, we have $$\Pi_{{\rm temp}}(H)=\coprod_{\varphi\in \Phi_{{\rm temp}}(H)}\Pi(\varphi,H).$$
\end{enumerate}

Suppose $H$ is a real form of equal rank. For each Langlands parameter $\varphi$ of $H$, Vogan associated it to a finite set $\Pi(\varphi)$ of irreducible representations of $\tilde{H}$ with $\tilde{H}$ running over all pure inner forms of $H$ (cf. \S \ref{VoganLpacket} or \cite{Vogan93}), called the Vogan L-packet associated to $\varphi$. It is known that there is a natural bijection between $\Pi(\varphi)$ and the set of irreducible representations of the finite group $A_\varphi=\pi_0(C_\varphi)$ (\cite[Theorem~6.3]{Vogan93}), where $C_\varphi$ is the centralizer of the image of $\varphi$ in $H^d$.

In \S \ref{subsec2.6}, we will describe such a bijection for limit of discrete series L-parameter $\varphi$ of $H$ \[\widehat{A_{\varphi}}\rightarrow \Pi(\varphi)\] which maps the identity of $\widehat{A_{\varphi}}$ to the generic representation in $\Pi(\varphi)$. Here we say a representation is generic if it admits a Whittaker model (\cite[\S 5]{Ada11}). 
\subsection{Main result}
let $\pi$ be a limit of discrete series representation of $G$ with the Langlands-Vogan parameter $(\varphi,\eta)$. Suppose \[\varphi=\oplus_{i=1}^kc_i\rho_{\lambda_i}\bigoplus (2z+1)\mathbf{1},\] where 
\begin{enumerate}
\item $\lambda_i$ are odd positive integers such that $\lambda_1>\cdots>\lambda_k>0$;
\item $\rho_{\lambda_i}$'s are self-dual irreducible representations of the Weil group $W_{\mathbb{R}}$ of dimension $2$;
\item $c_i, z \in {\bb N}$, $i=1,\cdots,k$, with $z + \sum_{i=1}^k c_i=n$.
\end{enumerate}
Then the component group $A_\varphi $ is a quotient of $\oplus_{i=1}^r ({\bb Z}/2{\bb Z})a_i \bigoplus \oplus_{j=1}^{z} ({\bb Z} /2{\bb Z}) b_j$ with $r=\sum_{i=1}^{k}c_i$. The signature of the character $\eta$ gives a partition of $c_i=p_{\eta,i}+q_{\eta, i}$, where $p_{\eta, i}$ and $q_{\eta,i}$ denote the number of positive signatures and negative signatures at the component corresponding to $\rho_{\lambda_i}$ respectively. 
\begin{theorem}\label{main1}
Let $V$ be a $2n$-dimensional symplectic space over $\mathbb{R}$.

$(1)$ Let $\pi$ be a limit of discrete series representation of $\mathrm{Sp}(V)$ with Langlands-Vogan parameter $(\varphi,\eta)$, where \[\varphi=\oplus_{i=1}^k(p_{\eta,i}+q_{\eta,i})\rho_{\lambda_i}\bigoplus (2z+1)\mathbf{1},\] with $p_{\eta,i},q_{\eta,i},z\in\mathbb{N}$, $i=1,\cdots,k$, $z + \sum_{i=1}^k (p_{\eta,i}+q_{\eta,i})=n$ and $\rho_{\lambda_i}$ self-dual irreducible representation of the Weil group $W_{\mathbb{R}}$ of dimension 2 with $\lambda_i$ an odd integer. Then there exists a unique pair of even integers $(p,q)$ satisfying $p+q=2n+2$ and a $(2n+2)$-dimensional orthogonal space $V^{'}$ with signature $(p,q)$ such that $\theta_{p,q}(\pi)$ is a limit of discrete series representation of $\mathrm{O}(V^{'})$ with Langlands-Vogan parameter ($\theta_{p,q}(\varphi)$, $\theta_{p,q}(\eta))$, where \[\theta_{p,q}(\varphi)=\varphi+{\bf 1}=\oplus_{i=1}^k(p_{\eta,i}+q_{\eta,i})\rho_{\lambda_i}\bigoplus(2z+2)\mathbf{1}.\]
Then the component group $A_{\varphi}$ is a subgroup of $A_{\theta_{p,q}(\varphi)}$, and we have
$$\theta_{p,q}(\eta)|_{A_\varphi}=\eta.$$
 
$(2)$ Let $V^{'}$ be a $(2n+2)$-dimensional real orthogonal space with signature $(p,q)$, where $p,q$ are even integers, and let $\pi^{'}$ be a limit of discrete series representation of $\mathrm{O}(V^{'})$ with Langlands-Vogan parameter $(\varphi^{'},\eta^{'})$. Assume that $\theta_{V^{'},V}(\pi^{'})\neq 0$. Then the Langlands-Vogan parameter $(\varphi,\eta)$ of the representation $\theta_{V^{'},V}(\pi^{'})$ of $\mathrm{Sp}(V)$ is given by
\[\varphi=\varphi^{'}-\mathbf{1}, \quad \eta=\eta^{'}\vert_{A_{\varphi}}.\]

\end{theorem}

Moreover, by the Langlands-Vogan
parametrization of parabolic inductions and the induction principle
of theta lifts, the main result of the tempered case can be reduced to the
limit of discrete series case.

\begin{theorem}\label{main2}
Let $V$ be a $2n$-dimensional symplectic space over $\mathbb{R}$.
    
$(1)$ Let $\pi$ be a tempered representation of $\mathrm{Sp}(V)$ with Langlands-Vogan parameter $(\varphi,\eta)$. Then there exists a unique pair of even integers $(p,q)$ satisfying $p+q=2n+2$ and a $(2n+2)$-dimensional orthogonal space $V^{'}$ with signature $(p,q)$ such that $\theta_{p,q}(\pi)$ is a tempered representation of $\mathrm{O}(p,q)$ with Langlands-Vogan parameter ($\theta_{p,q}(\varphi)$, $\theta_{p,q}(\eta))$, where $\theta_{p,q}(\varphi)=\varphi+{\bf 1}$. Then the component group $A_{\varphi}$ is a subgroup of $A_{\theta_{p,q}(\varphi)}$, and we have
$$\theta_{p,q}(\eta)|_{A_\varphi}=\eta.$$

$(2)$ Let $V^{'}$ be a $(2n+2)$-dimensional real orthogonal space with signature $(p,q)$, where $p,q$ are even integers, and let $\pi^{'}$ be a tempered representation of $\mathrm{O}(V^{'})$ with Langlands-Vogan parameter $(\varphi^{'},\eta^{'})$. Assume that $\theta_{V^{'},V}(\pi^{'})\neq 0$. Then the Langlands-Vogan parameter $(\varphi,\eta)$ of the tempered representation $\theta_{V^{'},V}(\pi^{'})$ of $\mathrm{Sp}(V)$ is given by
\[\varphi=\varphi^{'}-\mathbf{1}, \quad \eta=\eta^{'}\vert_{A_{\varphi}}.\]
\end{theorem}

\begin{remark}
The corresponding reinterpretation for real unitary groups case has been done in \cite[Section 5.3.1]{Ichi20}; for metaplectic-orthogonal dual pairs, the reinterpretation is trivial since Gan, Gross and Prasad \cite{GGP12} defined the Langlands-Vogan parameters of metaplectic groups via theta correspondence. 
\end{remark}

\subsection{Acknowledgement} This note serves as a preparation for a forthcoming paper, initially started from a discussion with Hang Xue. 
During the preparation of this note, Shanwen Wang has benefited by the discussions with Wenwei Li, Cai Li, David Renard, Hang Xue and Lei Zhang. Part of this article is written during the visits of first author at BIMCR, Peking University and HongKong University. The first author would like to thank Wenwei Li and Kei Yuen Chan for their hospitalities. The third author would like to thank Jeffrey Adams for helpful discussions about the material to be found in Section \ref{Realclassicalgroup}. Finally, the authors would like to express their special gratitude to Hang Xue for his constant support. 
\section{Preliminary}

We consider the following real classical groups of rank $n$ with $n\geq 1$:
\begin{enumerate}
\item The symplectic group ${\rm Sp}_{2n}$. Its Langlands dual group is ${\rm SO}_{2n+1}(\mathbb{C})$ and its $L$-group is the direct product ${\rm SO}_{2n+1}(\mathbb{C})\times W_{\mathbb{R}}$.

\item The even split special orthogonal group ${\rm SO}_{2n}^{\rm s}$, $n\geq 2$. Its Langlands dual group is ${\rm SO}_{2n}(\mathbb{C})$ and its L-group is the direct product ${\rm SO}_{2n}(\mathbb{C})\times W_{\mathbb{R}}$.
\item The quasi-split even special orthogonal group ${\rm SO}_{2n}^{\rm qs}$, $n\geq 1$. Its Langlands dual is ${\mathrm{SO}}_{2n}(\mathbb{C})$ and its $L$-group is the semi-direct product ${\rm SO}_{2n}(\mathbb{C})\rtimes W_{\mathbb{R}}$.
\end{enumerate}

\subsection{Complex reductive groups}

\subsubsection{Based Root Datum}
Let ${\bf G}$ be a complex reductive group. A Borel pair $({\bf B}, {\bf T})$ of ${\bf G}$ is a pair consisting of a maximal torus ${\bf T}$ of ${\bf G}$ and a Borel subgroup ${\bf B}$ containing ${\bf T}$.
For any Borel pair $({\bf B}, {\bf T})$ of ${\bf G}$, let $X({\bf T})$ be the group of characters of ${\bf T}$, $\Phi$ the set of roots, $X^\vee({\bf T})$ the group of cocharacters of ${\bf T}$ and $\Phi^\vee$ the set of coroots.
We denote the natural pairing $X({\bf T})\times X^\vee({\bf T})\rightarrow {\bb Z}$ by $\langle\cdot, \cdot\rangle$.
We also fix the canonical isomorphisms:
\[\mathfrak{t}\cong X^\vee({\bf T}) \otimes_{\bb Z} {\bb C},\quad \mathfrak{t}^{\vee} \cong X({\bf T})\otimes_{\bb Z} {\bb C},\] where $\mathfrak{t}=\mathrm{Lie}({\bf T}).$

\begin{definition}
\begin{enumerate}
    \item The root datum of ${\bf G}$ is a $4$-tuple $$R({\bf G}, {\bf T})=(X(\mathbf{T}), \Phi, X(\mathbf{T})^\vee, \Phi^{\vee})$$ where $\Phi$ (resp. $\Phi^\vee$) is the set of roots (resp. coroots) associated to the pair $(\mathbf{G},\mathbf{T})$.
\item Fix a Borel pair $({\bf B}, {\bf T})$ of ${\bf G}$. Let $\Delta$ (resp. $\Delta^\vee$) be the set of positive simple roots (resp. coroots) corresponding to ${\bf B}$. We call the $6$-tuple $D_b=(X,\Phi,\Delta, X^{\vee},\Phi^{\vee},\Delta^{\vee})$ the based root datum of $\mathbf{G}$ associated to $(\mathbf{B},\mathbf{T})$.
\end{enumerate}

\end{definition}

Up to a canonical isomorphism, the $6$-tuple is independent of the choice of Borel pair: if $({\bf B}^{'}, {\bf T}^{'})$ is another Borel pair of ${\bf G}$,
then there exists $g\in {\bf G}$ such that ${\rm Ad}(g)$ carries $({\bf B}^{'}, {\bf T}^{'})$ to $({\bf B}, {\bf T})$ and the induced isomorphism on based root datum $D_b$ of ${\bf G}$ is independent of $g$.

\subsubsection{Automorphisms} We denote by ${\rm Inn}({\bf G})$, ${\rm Aut}({\bf G})$ and ${\rm Out}({\bf G})$, the group of inner automorphisms of ${\bf G}$, the group of (holomorphic) automorphisms of ${\bf G}$ and the group of outer automorphisms of ${\bf G}$ respectively.
There exists a short exact sequence of groups
\begin{align}\label{short}
    1\rightarrow {\rm Inn}({\bf G})\rightarrow {\rm Aut}({\bf G})\rightarrow {\rm Out}({\bf G})\rightarrow 1.
\end{align}
The group ${\rm Aut}({\bf G})$ acts on the set of Borel pairs of ${\bf G}$ with maximal split torus ${\bf T}$.
If $\sigma\in{\rm Aut}({\bf G})$, the Borel pairs $(\sigma({\bf B}), \sigma({\bf T}))$ and $({\bf B}, {\bf T})$ are conjugate by an element $g_\sigma \in {\bf G}$, which is uniquely determined by $\sigma$ up to an element of ${\bf T}$ (cf. \cite[Prop. 6.2.11 (2)]{Con14}).
This induces a group homomorphism ${\rm Aut}({\bf G})\rightarrow {\rm Aut}(R({\bf G}, {\bf T}), \Delta)$ defined by $\sigma\mapsto {\rm Ad}(g_\sigma)\circ\sigma$, where ${\rm Aut}(R({\bf G}, {\bf T}),\Delta)$ is the group consisting of automorphisms of $D_b$.
By \cite[Prop. 7.1.6]{Con14}, there is an exact sequence
\begin{equation}\label{pinning}
1\rightarrow {\rm Inn}({\bf G})\rightarrow {\rm Aut}({\bf G})\rightarrow {\rm Aut}(R({\bf G}, {\bf T}),\Delta)\rightarrow 1,
\end{equation}
identifying ${\rm Out}({\bf G})$ with ${\rm Aut}(R({\bf G}, {\bf T}),\Delta)$.
\subsubsection{Pinning and L-group}\label{pinningdef}
\begin{definition}
    A pinning for ${\bf G}$ is a triple ${\bf Spl}_{{\bf G}}=({\bf B}, {\bf T}, \{X_\alpha\}_{\alpha\in \Delta})$, where $({\bf B}, {\bf T})$ is a Borel pair of ${\bf G}$, $\Delta$ is the set of simple roots associated to $({\bf B}, {\bf T})$ and $X_\alpha$ is a $\alpha$-root vector of ${\bf T}$ in ${\rm Lie}({ \bf B})$.\end{definition}

The group ${\bf G}$ acts on the set of pinnings by conjugation. Given a pinning ${\bf Spl}_{{\bf G}}$, we can associate it with an isomorphism
\[s_{{\bf Spl}_{{\bf G}}}:{\rm Out}({\bf G}) \cong {\rm Stab}_{{\rm Aut}({\bf G})}({\bf Spl}_{{\bf G}})\subset {\rm Aut}({\bf G})\]
and this is a splitting of the exact sequence (\ref{short}).
We call a splitting of the exact sequence (\ref{short}) distinguished if it fixes a pinning for ${\bf G}$.

The complex dual group ${\bf G}^d$ of ${\bf G}$ is the complex, connected reductive group whose root datum is isomorphic to $R({\bf G}, {\bf T})^{\vee}$.
Fix such an isomorphism of $R({\bf G}, {\bf T})^{\vee}$ with $R({\bf G}^{d}, {\bf T}^{d})$.
Through the natural Galois action $\Gamma$ on $R({\bf G}, {\bf T})$, we obtain a homomorphism $\Gamma \rightarrow {\rm Out}({\bf G}^d)$, using the identification of two short sequences (\ref{short}) and (\ref{pinning}).
By composing with the section $s_{{\bf Spl}_{{\bf G}^d}}$ defined by a pinning ${\bf Spl}_{{\bf G}^d}$, we obtain an action of $\Gamma$ on ${\bf G}^d$ which preserves the pinning ${\bf Spl}_{{\bf G}^d}$.
The Langlands group $^{L}{\bf G}$ of ${\bf G}$ associated with the pinning ${\bf Spl}_{{\bf G}^d}$ is the semi-direct product ${\bf G}^d\rtimes W_{\bb R}$,
where the action of $W_{\bb R}$ on ${\bf G}^d$ factors through the projection $p_{W_{\bb R}}: W_{\bb R}\rightarrow \Gamma$ and it stabilizes ${\bf Spl}_{{\bf G}^d}$.

\subsection{Root systems of real reductive groups}

In the study of connected semisimple groups up to central isogeny, it is convenient to work with a coarser notion than a root datum, in which we relax the $\mathbb{Z}$-structure to a $\mathbb{Q}$-structure and remove the explicit mention of the coroots. 

\begin{definition}
    A root system is a pair $(V, \Phi)$ consisting of a finite-dimensional $\mathbb{Q}$-vector space $V$ and a finite spanning set $\Phi\subset V-\{0\}$ such that for each $\alpha\in\Phi$ there exists a reflection $s_{\alpha}: v \mapsto v-\lambda(v)\alpha$ with $\lambda\in V^*$ such that $s_\alpha(\Phi)=\Phi,s_\alpha(\alpha)=-\alpha$ and $\lambda(\Phi)\subset \mathbb{Z}$.
\end{definition}

\begin{remark}
If $(X,\Phi,X^\vee,\Phi^\vee)$ is a root datum, then the $\mathbb{Q}$-span $V$ of $\Phi$ together with $\Phi$ is a root system.
\end{remark}

Let ${\bf G}$ be a complex reductive group. A real form of a complex reductive group ${\bf G}$ is an antiholomorphic involutive automorphism $\sigma$ of ${\bf G}$.

\begin{definition} \cite[D\'efinition 5.1]{MD19} \label{FBP}Let $({\bf G}, \sigma_G)$ be a real form with complex reductive group ${\bf G}$. A Borel pair $({\bf B}, {\bf T})$ of a complex reductive group ${\bf G}$ is called fundamental if the following conditions are satisfied:

{\rm (i)} $T={\bf T}^{\sigma_G}$ is a maximally compact subgroup of $G={\bf G}^{\sigma_G}$;

{\rm (ii)} The set of roots of ${\bf T}$ in ${\bf B}$ is stable under $-\sigma_G$.

Moreover, a fundamental Borel pair $({\bf B}, {\bf T})$ of ${\bf G}$ is called of Whittaker type if all the imaginary simple roots of ${\bf T}$ in ${\bf B}$ are non-compact.
\end{definition}  
By \cite[Prop. 6.24]{AV92}, a real form $({\bf G}, \sigma_{G})$ has a fundamental Borel pair of Whittaker type if and only if $({\bf G}, \sigma_G)$ is quasi-split. This applies to real symplectic groups and real orthogonal groups of equal rank. 

Let $({\bf G}, \sigma_G)$ be a quasi-split real form and we fix a maximal torus $\mathbf{T}$ of ${\bf G}$. Let $\mathfrak{g}$ (resp. $\mathfrak{t}$) be the Lie algebra of $\mathbf{G}$ (resp. $\mathbf{T}$). Then we have a root system $\Phi(\mathfrak{g},\mathfrak{t})$ associated to the pair $(\mathfrak{g},\mathfrak{t})$. Choose a simple root system $\Delta(\mathfrak{g},\mathfrak{t})$ whose elements are non-compact roots. This choice determines a Borel subgroup $\mathbf{B}$ of $\mathbf{G}$. Moreover, the Borel pair $(\mathbf{B},\mathbf{T})$ is a fundamental Borel pair of Whittaker type of $\mathbf{G}$.

Let $\Psi$ be the system of positive roots generated by $\Delta$ and let $\Psi_{c}\subset \Psi$ be the subset of compact positive roots. We set 
\[\rho(\Psi)=\frac{1}{2}\sum\limits_{\alpha\in\Psi}\alpha \quad \text{ and } \quad \rho_c(\Psi)=\frac{1}{2}\sum\limits_{\alpha\in\Psi_{c}}\alpha .\] Note that $\rho_c(\Psi)$ is independent of the choice of the set of positive roots and we simply denote them by $\rho_c$ respectively. 

\subsubsection{Symplectic case}\label{HCSp} Let $G=\mathrm{Sp}_{2n}({\bb R})$ and let $({\bf G},\sigma_G)$ be the real form corresponding to $G$. Let $\mathfrak{g}$ be the Lie algebra of $\mathbf{G}$. The maximally compact subgroup of $G$ is 
\[T=\{\left(\begin{smallmatrix}{\rm diag}(\cos t_1, \cdots, \cos t_n)& {\rm diag}(\sin t_1, \cdots, \sin t_n)\\ {\rm diag}(-\sin t_1, \cdots, -\sin t_n)&{\rm diag}(\cos t_1, \cdots, \cos t_n)\end{smallmatrix}\right):t_i\in {\bb R},1\leq i\leq n\}.\] The Lie algebra of $T$ is 
\begin{equation}
\mathfrak{t}_0=\{\left(\begin{smallmatrix}0_n & {\rm diag}(t_1, \cdots,t_n)\\ {\rm diag}(-t_1,\cdots, -t_n)& 0_n\end{smallmatrix}\right): t_i\in {\bb R}, 1\leq i\leq n\}.
\end{equation}
Let $\mathfrak{t}$ be the complexification of $\mathfrak{t}_0$. For $1\leq i\leq n$, we define the characters of $\mathfrak{t}$:
\[e_i: \left(\begin{smallmatrix}0_n & {\rm diag}(t_1, \cdots,t_n)\\ {\rm diag}(-t_1,\cdots, -t_n)& 0_n\end{smallmatrix}\right)\in \mathfrak{t} \mapsto 2it_i\in {\bb C}.\] 
Then we have $\mathfrak{t}^*=\oplus_{i=1}^n\mathbb{C}e_i$ and the set of roots is
\[\Phi(\mathfrak{g}, \mathfrak{t})=\{\pm e_i\pm e_j: 1\leq i< j\leq n\}\cup \{ \pm 2e_i: 1\leq i\leq n \}.\]

Let $\Delta(\mathfrak{g}, \mathfrak{t})$ be the subset $\{e_i+e_{n+1-i}, -e_{n+1-i}-e_{i+1} : 1\leq i \leq n-1 \}$ of $\Phi(\mathfrak{g}, \mathfrak{t})$, which is a set of non-compact simple roots. We denote by $\Psi$ the set of positive roots generated by $\Delta(\mathfrak{g}, \mathfrak{t})$.
Moreover, the subset of positive compact roots of $\Phi(\mathfrak{g}, \mathfrak{t})$ is $\Phi_c^+=\{e_i-e_j: 1\leq i< j\leq n\}$. This based root datum determines a fundamental Borel pair of Whittaker type of ${\bf G}$.

\subsubsection{Orthogonal case}\label{HCO} Let $G={\rm O}(p,q)$ with $p\geq q$ even. Let $({\bf G},\sigma_G)$ be the real form corresponding to $G$. Let $\mathfrak{g}$ be the Lie algebra of $\mathbf{G}$. Let $p_0=\frac{p}{2}$ and $q_0=\frac{q}{2}$. The maximally compact subgroup $T$ of $G$ consists of the following matrices:  \begin{equation*}    {\rm diag}((\begin{smallmatrix}\cos t_1 & \sin t_1\\ -\sin t_1& \cos t_1\end{smallmatrix}),\cdots,(\begin{smallmatrix}\cos t_{p_0} & \sin t_{p_0}\\ -\sin t_{p_0}& \cos t_{p_0}\end{smallmatrix}), \\(\begin{smallmatrix}\cos s_1 & \sin s_1\\ -\sin s_1& \cos s_1\end{smallmatrix}),\cdots,(\begin{smallmatrix}\cos s_{q_0} & \sin s_{q_0}\\ -\sin s_{q_0}& \cos s_{q_0}\end{smallmatrix})),
\end{equation*}
with $t_i,s_j\in {\bb R}$.

For $t\in {\bb R}$, we denote the matrix $(\begin{smallmatrix}0& t\\ -t& 0\end{smallmatrix})$ by $g(t)$. 
The Lie algebra $\mathfrak{t}_0$ of $T$
consists of matrices
$\{{\rm diag}(g(t_1),\cdots,g(t_{p_0}),g(s_1),\cdots, g(s_{q
_0}))$ with $ t_i, s_j\in {\bb R}$ for $1\leq i\leq p_0$ and $1\leq j\leq q_0$.
Let $\mathfrak{ t}$ be the complexification of $\mathfrak{t}_0$. For $1\leq i\leq p_0$ and $1\leq j\leq q_0$, we define the characters $e_i$ and $f_j$ of $\mathfrak{t}$ by the following laws:
\[e_i: {\rm diag}(g(t_1),\cdots,g(t_{p_0}),g(s_1),\cdots, g(s_{q
_0})) \mapsto 2it_i\in {\bb C}, \]
\[f_j: {\rm diag}(g(t_1),\cdots,g(t_{p_0}),g(s_1),\cdots, g(s_{q
_0}))\mapsto 2is_{j}\in {\bb C}.\] Then we have $\mathfrak{t}^*=\oplus_{i=1}^{p_0}\mathbb{C}e_i\bigoplus\oplus_{j=1}^{q_0}\mathbb{C}f_j$ and the set of roots is
\begin{equation*}
\begin{split}
\Phi(\mathfrak{ g}, \mathfrak{ t})=&\{\pm e_i\pm e_j: 1\leq i< j\leq p_0\}\cup \{ \pm f_i\pm f_j: 1\leq i<j\leq q_0\}\\
& \cup \{\pm e_i\pm f_j: 1\leq i\leq p_0, 1\leq j\leq q_0\}\end{split}
\end{equation*} 

Let $\Delta(\mathfrak{g},\mathfrak{t})$ be the subset of $\Phi(\mathfrak{g},\mathfrak{t})$ defined as follows:
\begin{enumerate}
\item If $p=q$ or $p=q+2$, then \[\Delta(\mathfrak{g}, \mathfrak{t})=\{e_i-f_i,f_i-e_{i+1}\vert 1\leq i\leq q_0 \}\cup\{ e_{p_0}+f_{q_0} \};\]
\item If $p>q+2$, then $\Delta(\mathfrak{g},\mathfrak{t})$ is a set of following non-compact simple roots: \[\{e_i-f_i,f_i-e_{i+1}\vert 1\leq i\leq q_0 \}\cup\{ e_j-e_{j+1}\vert q_0+1\leq j\leq p_0 \}\cup\{e_{p_0-1}+e_{p_0} \}.\]
\end{enumerate}

We denote by $\Psi$ the set of positive roots generated by $\Delta(\mathfrak{g},\mathfrak{t})$.
Moreover, the subset of positive compact roots $\Phi_c^+$ of $\Phi(\mathfrak{g},\mathfrak{t})$ is 
\[\{e_i-e_j: 1\leq i< j\leq p_0\}\cup\{f_i-f_j: 1\leq i< j\leq q_0\}.\] This based root datum determines a fundamental Borel pair of Whittaker type of ${\bf G}$.

\subsection{Generalities on real forms and Cartan involutions}\label{RealCartan}
We recall some structure theories of real reductive groups and the parameterization of real forms. Our main reference is \cite{AT18}.
\begin{definition}
Let ${\bf G}$ be a complex reductive group.
\begin{enumerate}
    \item We say that two real forms $\sigma_1$, $\sigma_2$ are inner to each other, or in the same {\it inner class}, if $\sigma_1\sigma_2^{-1}$ is an inner automorphism of ${\bf G}$.
    \item We say that two real forms $\sigma_1$, $\sigma_2$ are equivalent, if they are conjugate by an inner automorphism of ${\bf G}$. 
    \item A real form $\sigma$ of ${\bf G}$ is said to be a compact real form if ${\bf G}^{\sigma}$ is compact and meets every component of ${\bf G}$.
\end{enumerate}
\end{definition}
Given a real reductive group $G$, it is equivalent to provide a real form $\sigma_G$, which satisfies $(G\otimes_{\bb R} {\mathbb{C}})^{\sigma_G}=G$.

\begin{remark}\label{equivclH1}
The standard definition of equivalence of real forms (cf. \cite[Section III.1]{Serre02}) allows conjugation by ${\rm Aut}({\bf G})$. But since we are interested in the inner class, we follow the definition of Adams and Ta\"ibi in \cite{AT18}.
Moreover, for a real form $\sigma$ of ${\bf G}$, by \cite[Lemma 8.1]{AT18}, the set of equivalent classes of real forms in the inner class of $\sigma$ is parametrized by $H^1(\sigma, {\bf G}_{\rm ad})$, where ${\bf G}_{\rm ad}$ is the adjoint group.
Explicitly, the map is ${\rm cl}(h) \mapsto [{\rm int}(h)\circ \sigma]$.
By \cite[Lemma~2.4]{AT18}, for an equivalence class $[\sigma]$, we have a well-defined pointed set $H^1([\sigma],{\bf G}_{\rm ad})=H^1(\sigma,{\bf G}_{\rm ad})$.
\end{remark}
\begin{remark}\label{CptRealForm}
By \cite[Lemma 3.4]{AT18}, our definition of compact real form is equivalent to the definition of Mostow \cite[Section 2]{Mos55}, which defines a compact real form to be a compact subgroup $G_K$ of $\mathbf{G}$ such that $${\rm Lie}({\bf G})={\rm Lie}(G_K)\oplus i\cdot \mathrm{Lie}(G_K),$$ and $G_K$ meets every connected component of $\mathbf{G}$. The bijection is given by $\sigma \mapsto {\bf G}^{\sigma}$.
Every complex reductive group has a compact real form (Weyl, Chevalley, Mostow \cite[Lemma~6.2]{Mos55}), and the uniqueness is up to the ${\bf G}^0$ conjugation (Cartan, Hochschild, Mostow \cite[Ch. XV]{Ho65}, \cite[Theorem 3.1]{Mos55}), where ${\bf G}^0$ is the identity component of $\mathbf{G}$.
\end{remark}

Cartan involution provides a description of real forms in terms of holomorphic involutions which is better suited to our purposes.
\begin{definition}
A {\it Cartan involution} for $({\bf G}, \sigma)$, where $\sigma$ is a real form of complex reductive group ${\bf G}$, is a holomorphic involutive automorphism $\theta$ of ${\bf G}$, commuting with $\sigma$, such that $\theta\sigma$ is a compact real form of $\bf G$.
\end{definition}
The existence and uniqueness (up to conjugation by ${\rm Inn}({\bf G}^0)$) of Cartan involution, and the correspondence between
    \begin{equation*}
    \{\text{antiholomorphic involutive automorphisms of ${\bf G}$}\}/{\rm Inn}({\bf G}^0)
    \end{equation*}
    and 
    \begin{equation*}
    \{\text{holomorphic involutive automorphisms of ${\bf G}$}\}/{\rm Inn}({\bf G}^0)
    \end{equation*}
induced by the correspondence between real forms and Cartan involutions, are given by \cite[Theorem~3.13]{AT18}, based on Remark \ref{CptRealForm}.
The following construction also gives an explanation.

Fix a Borel pair $({\bf B}, {\bf T})$ of ${\bf G}$.
If $\sigma$ is a real form of ${\bf G}$ and $\theta$ is a Cartan involution for $({\bf G}, \sigma)$, then both $\sigma$ and $\theta$ naturally act on the based root datum $D_b$ attached to $({\bf B}, {\bf T})$, giving rise to two involutions $\overline{\sigma}, \overline{\theta}\in {\rm Aut}(R({\bf G},{\bf T}),\Delta)$,
which are also seen as elements of the subgroup ${\rm Out}({\bf G})[2]$ of order $2$ elements in ${\rm Out}({\bf G})$.
They are related by $\overline{\sigma}\overline{\theta}=-w_0$, where $w_0$ is the longest element of the Weyl group of $D_b$ and $-1$ is the inversion automorphism of ${\bf T}$.
Note that $w_0$ is invariant under ${\rm Aut}(R({\bf G},{\bf T}),\Delta)$, and so $\iota:=-w_0$ is a central involution in ${\rm Out}({\bf G})$.
As a result, we see that the set of inner classes of real forms of ${\bf G}$ can be parametrized by ${\rm Out}({\bf G})[2]$,
and we say that a real form $\sigma$ lies in the inner class defined by $\delta\in {\rm Out}({\bf G})[2]$ if $\iota\delta=\overline{\sigma}$. We also say its corresponding Cartan involution lies in the inner class defined by $\delta$.

\begin{definition}\label{equalrk}The real forms in the inner class defined by $1\in {\rm Out}({\bf G})[2]$ are called of equal rank.
\end{definition}

Every real form ${\bf G}(\bb R)$ in this inner class contains a compact Cartan subgroup, and every Cartan involution in this inner class is contained in ${\rm Inn}({\bf G})$.
The reason why this inner class is of our interest is that a necessary and sufficient condition for ${\bf G}(\bb R)$ to admit discrete series is that it has a compact Cartan subgroup, by \cite[Theorem~13]{Ha66}.

For an inner class $\delta\in {\rm Out}({\bf G})[2]$ and a pinning ${\bf Spl}_{{\bf G}}=({\bf B}, {\bf T}, \{X_\alpha\}_{\alpha\in \Delta})$, there is a unique real form $\sigma_{\rm qs}(\delta,{\bf Spl_G})$ of ${\bf G}$ preserving ${\bf Spl_G}$ and such that $\overline{\sigma_{\rm qs}(\delta,{\bf Spl_G})}=\iota\delta$, and it is naturally a quasi-split real form\footnote{A real form is called quasi-split if it preserves a Borel subgroup of ${\bf G}$.}.
Since for $g\in {\bf G}_{\rm ad}$, we have $\sigma_{\rm qs}(\delta,{\rm int}(g)({\bf Spl_G}))={\rm int}(g)\circ \sigma_{\rm qs}(\delta,{\bf Spl_G}) \circ {\rm int}(g)^{-1}$, the equivalence class of $\sigma_{\rm qs}(\delta,{\bf Spl_G})$ does not depend on the choice of ${\bf Spl_G}$. We denote this unique equivalence class of quasi-split forms in the inner class defined by $\delta$ as $[\sigma_{\rm qs}(\delta)]$. Denote its corresponding (inner class of) Cartan involution $[\theta_{{\rm qs}}(\delta)]$.

Conversely, the above construction contains all quasisplit real forms of ${\bf G}$.
The conjugacy class of quasi-split real forms has a particular interest as they can be characterized by Borel pairs with special properties.
\begin{definition}
We say a real form is {\it quasi-compact} if one (equivalently, any) of its Cartan involutions is distinguished.
\end{definition}
For each pinning ${\bf Spl_G},$ each inner class $\delta\in {\rm Out}({\bf G})[2]$ contains a unique quasi-compact real form $\sigma_{{\rm qc}}(\delta,{\bf Spl_G})$, whose corresponding (inner class of) Cartan involution is denoted by $\theta_{{\rm qc}}(\delta,{\bf Spl_G})$ (\cite[Section~5]{AC09}). Similarly, the equivalent class of $\sigma_{{\rm qc}}(\delta,{\bf Spl_G})$ does not depend on the choice of ${\bf Spl_G}$. We denote this unique class of quasi-compact forms in the inner class defined by $\delta$ as $[\sigma_{\rm qc}(\delta)]$. Denote its corresponding (inner class of) Cartan involution $[\theta_{{\rm qc}}(\delta)]$. Notice that an inner involution is distinguished if and only if it is the identity; this is the Cartan involution of the compact real form, i.e. $1\in [\theta_{\rm qc}(1)]$.

\begin{remark}\label{equivclH1forcartan}
Similarly to Remark \ref{equivclH1}, for a Cartan involution $\theta$ of ${\bf G}$, the set of equivalence classes of Cartan involutions in the inner class of $\theta$ is parametrized by $H^1(\theta, {\bf G}_{\rm ad})$, where ${\bf G}_{\rm ad}$ is the adjoint group.
Explicitly, the map is given by ${\rm cl}(h) \mapsto [{\rm int}(h)\circ \theta]$.
Again by \cite[Lemma~2.4]{AT18}, for equivalence class $[\theta]$, we have a well-defined pointed set $H^1([\theta],{\bf G}_{\rm ad})=H^1(\theta,{\bf G}_{\rm ad})$.

Furthermore, for the corresponding Cartan involution $\theta$ and real form $\sigma$, the following diagram commutes:
\[\begin{tikzcd}
	{H^1(\sigma,{\bf G}_{\rm ad})} & {\{\text{real forms inner to }\sigma}\} \\
	{H^1(\theta,{\bf G}_{\rm ad})} & {\{\text{Cartan involutions inner to }\theta}\}
	\arrow["\cong", from=1-1, to=1-2]
	\arrow["\cong", from=1-1, to=2-1]
	\arrow["\cong", from=1-2, to=2-2]
	\arrow["\cong", from=2-1, to=2-2]
\end{tikzcd}\]
Here ${H^1(\sigma,{\bf G}_{\rm ad})}\cong {H^1(\theta,{\bf G}_{\rm ad})}$ is the canonical bijection of pointed sets (\cite[Lemma~8.4]{AT18}).
\end{remark}

\section{Parameters for real reductive group of equal rank}

Let $G$ be a real reductive group of equal rank. Let $\mathfrak{g}$ be the Lie algebra of $G_{\bb C}$ and $K$ a maximal compact subgroup of $G_{\bb C}$. We fix a fundamental Borel pair $({\bf B}_{*}, {\bf T}_{*})$ of Whittaker type and the based root datum $(X, \Phi, \Delta_{*}, X^{\vee}, \Phi^{\vee}, \Delta_{*}^{\vee})$ associated to $({\bf B}_{*}, {\bf T}_{*})$. Let $\mathfrak{t}$ be the Lie algebra of $\mathbf{T}_{*}$ and $\mathfrak{t}_0$ be the Lie algebra of $\mathbf{T}_*\cap G$. Let $\Psi_{*}$ be the set of positive roots generated by $\Delta_{*}$ and $\Psi_{c,*}$ be its subset of compact positive roots. For $\lambda\in \mathfrak{t}^*$, we say $\lambda$ is regular (resp. integral) if $\langle\lambda, \alpha^{\vee}\rangle\in {\bb Z}\setminus \{0\}$ (resp. is in ${\bb Z}$) for all $\alpha\in \Delta_*$. Let $\rho_c$ be the half sum of the roots in $\Psi_{c,*}$.

\subsection{Harish-Chandra parameters}

Let $\mathrm{Rep}(G)$ be the set of irreducible representations of $G$. For any $\pi\in \mathrm{Rep}(G)$, it has a $(\mathfrak{g},K)$-module structure. As a result, we have following two data attached to $\pi$:

(1) There is an ``infinitesimal character map''
\[ \mu: \mathrm{Rep}(G)\rightarrow {\rm Hom}_{{\bb C}-{\rm alg}}(Z(\mathfrak{ g}), {\bb C}),\pi\mapsto \mu_\pi,\]
and by Harish-Chandra's finiteness theorem \cite[Theorem~5.5.6]{Wal88}, this map has finite fibres. By Harish-Chandra homomorphism
\footnote{The Poincare-Birkhoff-Witt theorem implies that we have a decomposition of $U(\mathfrak g)$:
\[U(\mathfrak{ g})=U(\mathfrak{ t})\oplus (U(\mathfrak{ g}) \mathfrak{ n}^++\mathfrak{ n}^{-}U(\mathfrak g)),\] and the projection of $Z(\mathfrak{g})$ to the second factor lies in $U(\mathfrak{ g}) \mathfrak{n}^+\cap\mathfrak{ n}^{-}U(\mathfrak g)$. Let $\gamma^{'}: Z(\mathfrak{ g})\rightarrow Z(\mathfrak{ t})$ be the projection on to the first factor. Let $\rho$ be the half the sum of the positive roots associated to $\Phi$. Let $t_{\rho}: \mathfrak{ t}\rightarrow U(\mathfrak{ t})$ be the translation operator $t_{\rho}(h)=h-\rho(h)1$. The composition $\gamma := t_{\rho}\circ\gamma^{'}: Z(\mathfrak{ g})\rightarrow U(\mathfrak{ t})$ is a homomorphism, known as the Harish-Chandra homomorphism. The image of the Harish-Chandra homomorphism is invariant under the action of Weyl group, and the map is actually an isomorphism $\gamma : Z(\mathfrak{ g})\rightarrow U(\mathfrak{t})^W = S(\mathrm{t})^W = {\bb C}[\mathfrak{ t}^*]^W$. }, we can lift the infinitesimal character $\mu_\pi$ to a character of $\mathfrak{ t}$ and such a lifting is not unique. We say an infinitesimal character is regular (resp. integral) if one of its liftings is regular (resp. integral). Note that if one lifting of $\mu_\pi$ is regular (resp. integral), then all the liftings of $\mu_\pi$ are regular (resp. integral). We say $\pi$ is a limit of discrete series if its infinitesimal character is integral. 

(2) The restriction of $\pi$ to $K$ can be decomposed into a completed direct sum of irreducible representations of $K$. An irreducible representation $\tau$ of $K$ appearing in the decomposition of $\pi|_K$ is called a $K$-type of $\pi$, which can be parameterized by its highest weight $\mu_{\tau}\in \mathfrak{t}^*$. Vogan defined a norm on the set of $K$-types of $\pi$ as following:
\[\Vert \tau\Vert:=\sqrt{\langle \mu_\tau+2\rho_c, \mu_{\tau}+2\rho_c \rangle}. \] The minimal $K$-type of $\pi$ is a $K$-type of $\pi$ with minimal norm among all $K$-types of $\pi$ (cf. \cite[Definition 5.1]{Vogan79}).

Let $\pi$ be an irreducible representation of $G$ with integral infinitesimal character $\mu_{\pi}$. 
Suppose $\tau_0$ is the minimal $K$-type of $\pi$. Consider the set $Y$ of pairs $(\lambda,\Psi)$, where $\lambda\in i\mathfrak{t}_0^{*}\subset \mathfrak{ t}^{*}$ is a lifting of $\mu_\pi$, and $\Psi\subset \Phi$ is the set of positive roots with respect to $\lambda$ satisfying:
 \begin{enumerate}[(a)]
\item $\Psi_{c,*}\subset \Psi$;
\item $\lambda$ is dominant with respect to  $\Psi$;
\item if a simple root $\alpha\in \Psi$ satisfies $\langle \lambda, \alpha^{\vee}\rangle=0$, then $\alpha$ is non-compact. 
\end{enumerate}
Then by \cite[Corollary 3.44]{Vogan84}, there is only one pair $(\lambda_\pi, \Psi_\pi)\in Y$ satisfying the equation \[ \mu_{\tau_0} = \lambda_\pi + \rho(\Psi_\pi) - 2\rho_c,\] where $\rho(\Psi_\pi)$ is the half sum of the roots in $\Psi_\pi$. We call the triple $(\lambda_\pi,\Psi_\pi,\mu_{\tau_0})$ the Harish-Chandra parameter of $\pi$.

\begin{definition}
A limit of discrete series Harish-Chandra parameter of $G$ is a pair $(\lambda_d, \Psi)$, where $\lambda_d\in i\mathfrak{t}_0^{*}\subset \mathfrak{ t}^{*}$ is integral, and $\Psi\subset \Phi$ is the set of positive roots with respect to $\lambda_d$ satisfying:
 \begin{enumerate}[(a)]
\item $\Psi_{c,*}\subset \Psi$;
\item $\lambda_d$ is dominant with respect to  $\Psi$;
\item if a simple root $\alpha\in \Psi$ satisfies $\langle \lambda_d, \alpha^{\vee}\rangle=0$, then $\alpha$ is non-compact. 
\end{enumerate}

\end{definition}

\begin{remark}
    Given a limit of discrete series Harish-Chandra parameter $(\lambda_d,\Psi)$ of $G$, if $\lambda_d$ is regular, then conditions $(a)$, $(b)$ and $(c)$ uniquely determine the set $\Psi$. In this case, we call $\lambda_d$ a discrete series Harish-Chandra parameter of $G$. If $\lambda_d$ is not regular, there are several possible $\Psi$'s admitting the same conditions $(a)$, $(b)$ and $(c)$. 
\end{remark}

\subsubsection{Symplectic case}
The explicit description of the Harish-Chandra parameters of the representations of real symplectic groups can be given as follows.
Let $\Psi$ and $\Delta$ be the sets of roots and simple roots of $G= {\rm Sp}_{2n}(\mathbb{R})$ defined in \S \ref{HCSp} respectively and let $\pi$ be a limit of discrete series representation of $G$.
Then the Harish-Chandra parameter $(\lambda_\pi, \Psi_\pi,\mu_\pi)$ of $\pi$, where
\begin{enumerate}
\item[$\bullet$] $\mu_\pi$ is the highest weight of the minimal $K$-type of $\pi$ and \[\mu_{\pi} = \lambda_\pi +\rho(\Psi_\pi) -2 \rho_c; \]
\item[$\bullet$] Note that $\mathfrak{t}^*=\oplus_{i=1}^n \mathbb{C}e_i$, then the parameter $\lambda_\pi$ with respect to the basis $\{e_i\}_{1\leq i\leq n}$ is of the form
\[\lambda_\pi=(\underbrace{\lambda_1,\cdots,\lambda_1}_{p_1},\cdots,\underbrace{\lambda_k,\cdots,\lambda_k}_{p_k},\underbrace{0,\cdots,0}_z,\underbrace{-\lambda_k,\cdots,-\lambda_k}_{q_k},\cdots,\underbrace{-\lambda_1,\cdots,-\lambda_1}_{q_1}),\]
with $\lambda_i\in\mathbb{Z}$, $\lambda_1>\cdots>\lambda_k>0$, $\vert p_i-q_i\vert \leq 1$;
\item[$\bullet$] $\Psi_\pi\subset \Phi$ is a root system containing all the positive compact roots, such that $\lambda_\pi$ is dominant with respect to $\Psi_\pi$, and for all simple roots $\alpha \in \Psi_\pi$ we have that if $\langle\lambda_\pi,\alpha \rangle=0$, then $\alpha$ is non-compact.

\item[$\bullet$] If $z=0$ and $p_i+q_i=1$ for all $i$, then the representation associated to $(\lambda_d,\Psi)$ is a discrete series. 
\end{enumerate}

\subsubsection{Orthogonal case}In the following, we recall the Harish-Chandra parametrization for representations of the orthogonal group of equal rank ${\rm O}(p,q)$, with $p$ and $q$ two non-negative integers. We will parametrize the representations of ${\rm O}(p,q)$ via the parameters of its maximal compact subgroup $K={\rm O}(p)\times {\rm O}(q)$. The equal rank condition implies that $p$ and $q$ are even if $p+q$ is even. 

Set $p_0=[\frac{p}{2}]$ and $q_0=[\frac{q}{2}]$.
We can parametrize an irreducible representation of compact group $\mathrm{O}(p)$ by $(\lambda_0;\epsilon)$, here $\lambda_0=(a_1,\cdots,a_{p_0})$ is the usual highest weight of a finite dimensional representation of $\mathrm{SO}(p)$ and $\epsilon\in\{\pm 1\}$. If $p$ is even and $a_{p_0}>0$, $(\lambda_0;1)$ and $(\lambda_0;-1)$ correspond to the same representation of $\mathrm{O}(p)$. If $p$ is odd then $-\mathrm{Id}$ acts by $(-1)^{\sum_{i=1}^{p_0}a_i}\epsilon$. If $p$ is even, the parameter of the trivial representation of $\mathrm{O}(p)$ corresponds to $(0,\cdots,0;1)$, the sign representation of $\mathrm{O}(p)$ corresponds to $(0,\cdots,0;-1)$, and we have $(a_1,\cdots,a_{[\frac{p}{2}]};\epsilon)\otimes\mathrm{sgn}=(a_1,\cdots,a_{[\frac{p}{2}]};-\epsilon)$. The representations of $\mathrm{O}(q)$ can be parametrized in the same way. Hence an irreducible finite dimensional representation of $K$ is parametrized by $(a_1,\cdots,a_{p_0};\epsilon)\otimes(b_1,\cdots,b_{q_0};\epsilon^{'})$. We will refer to $(a_1,\cdots,a_{p_0};b_1,\cdots,b_{q_0})$ as the highest weight, and to $(\epsilon,\epsilon')$ as the signs of the $K$-type.

Let $J=\mathrm{diag}(1,\cdots,1,-1)\in \mathrm{O}(p,q)\setminus \mathrm{SO}(p,q)$ and let $\sigma$ be the automorphism of $\mathrm{O}(p,q)$ giving by the conjugation by $J$. Note that $\sigma$ also acts on representations, Cartan subgroups and Lie algebra of $\mathrm{O}(p,q)$. Let $\pi$ be a limit of discrete of series representation of $\mathrm{O}(p,q)$. Then the restriction of $\pi$ to $\mathrm{SO}(p,q)$ is: 
\begin{equation*}
	\pi\vert_{\mathrm{SO}(p,q)}= \left\{\begin{smallmatrix}
		\pi, &\text{ if } \pi\vert_{\mathrm{SO}(p,q)} \text{ is irreducible};\\
		\pi_0\oplus\sigma(\pi_0), & \text{ otherwise.}
	\end{smallmatrix}\right.
\end{equation*} 
Here $\pi_0$ is an irreducible admissible representation of $\mathrm{SO}(p,q)$ with $\pi_0$ and $\sigma(\pi_0)$ are non-equivalent. In both cases, we choose the irreducible subrepresentation of $\pi\vert_{\mathrm{SO}(p,q)}$ with Harish-Chandra parameter $(\lambda_{\pi},\Psi_{\pi})$ described as follows: let $\Phi,\Delta$ be the sets of roots and simple roots of ${\rm O}(p,q)$ respectively defined in \S \ref{HCO}. Note that $\mathfrak{t}^*=\oplus_{i=1}^{p_0}\mathbb{C}e_i\bigoplus\oplus_{j=1}^{q_0}\mathbb{C}f_j$. 
\begin{enumerate}
	\item[$\bullet$] The parameter $\lambda_\pi$ with respect to the basis $\{e_i,f_j\}_{1\leq i\leq p_0,1\leq j\leq q_0}$ is of the form
	\[\lambda_\pi=(\underbrace{\lambda_1,\cdots,\lambda_1}_{p_1},\cdots,\underbrace{\lambda_k,\cdots,\lambda_k}_{p_k},\underbrace{0,\cdots,0}_z,\underbrace{\lambda_1,\cdots,\lambda_1}_{q_1},\cdots,\underbrace{\lambda_k,\cdots,\lambda_k}_{q_k},\underbrace{0,\cdots,0}_{z^{'}}),\] with 
	\begin{enumerate}
		\item $\lambda_k\in\mathbb{Z}$, $\lambda_1>\cdots>\lambda_k>0$, $\vert p_i-q_i\vert \leq 1$, 
		\item $\vert z-z^{'}\vert\leq 1$, 
		\item $p_0=\sum_{i=1}^{k} p_i+z$ and $q_0=\sum_{i=1}^{k}q_i+z^{'}$;
	\end{enumerate}
	\item[$\bullet$] $\Psi_\pi\subset \Phi$ is a root system containing all the positive compact roots, such that $\lambda_\pi$ is dominant with respect to $\Psi_\pi$.
\end{enumerate}

To parametrize $\pi$, we need one more parameter $\xi\in\{\pm 1\}$ called the sign of $\pi$. 
\begin{enumerate}
	\item[$\bullet$] If $\pi\vert_{\mathrm{SO}(p,q)}$ is irreducible, then $\pi\vert_{\mathrm{SO}(p,q)}$ has two liftings $\pi$ and $\pi\otimes\mathrm{det}$ to $\mathrm{O}(p,q)$. We choose the sign $+1$ for the representation $\pi$ and $-1$ for $\pi\otimes\mathrm{det}$; 
	\item[$\bullet$] If $\pi\vert_{\mathrm{SO}(p,q)}$ is reducible,  there is only one representation $\pi$ of $\mathrm{O}(p,q)$ whose restriction to $\mathrm{SO}(p,q)$ is $\pi_0\oplus\sigma(\pi_0)$. In this case the sign is arbitrary, we choose the sign $+1$.
	
\end{enumerate}

The triple $(\lambda_\pi,\xi,\Psi_\pi)$ is called the Harish-Chandra parameter of $\pi$.

\subsection{Real reductive groups of equal rank}\label{Realclassicalgroup}
We are interested in the real reductive groups of equal rank\footnote{The real form in the equal rank case is called pure inner form in the sense of Kaletha (cf. \cite{Kal16}).}, since the real forms admit discrete series representations if and only if it is of equal rank. Note that, if $n$ is odd (resp. even), then the split (resp. quasi-split) orthogonal groups of rank $n$ are not of equal rank. Thus for the even special orthogonal group, we will take the quasi-split (resp. split) orthogonal groups of rank $n$ when $n$ is odd (resp. even).

Through the general study on the strong real forms, we fix an inner class $\delta\in {\rm Out}({\bf G})[2]$. Consider the action of ${\rm Aut}({\bf G})$ on the center ${\bf Z}:=Z({\bf G})$ of ${\bf G}$, which factors through the quotient ${\rm Out}({\bf G})$.
We denote by ${\bf Z}^{\delta}$ the subgroup of ${\bf Z}$ fixed by $\delta$.

Recall that in \S \ref{RealCartan}, we denoted the unique inner class of quasi-split form (resp. quasi-compact form) defined by $\delta$, by $[\sigma_{{\rm qs}}(\delta)]$ (resp. $[\sigma_{{\rm qc}}(\delta)]$), together with their corresponding inner class of Cartan involution by $[\theta_{{\rm qs}}(\delta)]$ (resp. $[\theta_{{\rm qc}}(\delta)]$). For a fixed pinning ${\bf Spl_G}=\{B,H, \{X_\alpha\}\}$, we denoted the unique quasi-split form (resp. quasi-compact form) defined by $\delta$, by $\sigma_{{\rm qs}}(\delta,{\bf Spl_G})$ (resp. $\sigma_{{\rm qc}}(\delta,{\bf Spl_G})$), together with their corresponding Cartan involution by $\theta_{{\rm qs}}(\delta,{\bf Spl_G})$ (resp. $\theta_{{\rm qc}}(\delta,{\bf Spl_G})$).

In this section, we start with two perspectives of definitions on strong involutions and strong real forms, and explain how these two perspectives coincide. In particular, in the equal rank case, this allows us to use the set $\{h\in {\bf H}: h^2=z\}/W$, where ${\bf H}$ is a Cartan subgroup of ${\bf G}$ with Weyl group $W$, to parametrize the equivalence classes of strong real forms of ${\bf G}$ (Proposition \ref{srf1}).
\subsubsection{Strong Involution}
Denote $\Gamma=\{1,\theta_{{\rm qc}}(\delta,{\bf Spl_G})\}$. Define the extended group for $({\bf G},\delta)$ as the semi-direct product
$${\bf G}^\Gamma = {\bf G} \rtimes \Gamma.$$

\begin{definition}\cite[Definition 2.13]{ABV91}\label{SIgamma}
\begin{enumerate}
\item A {\it strong involution} of ${\bf G}$ in the inner class defined by $\delta$ is an element $\xi \in {\bf G}^\Gamma\setminus {\bf G}$ such that $\xi^2\in {\bf Z}_{\rm tor}$. The set of such strong involutions is denoted by ${\rm SI}_{\delta,{\bf Spl_G}}({\bf G})$.
Moreover, to a strong involution $\xi \in {\rm SI}_{\delta,{\bf Spl_G}}({\bf G})$, we can associate a central invariant
$${\rm Inv}(\xi)=\xi^2 \in {\Z}_{\rm tor}^{\delta}.$$

\item Two strong involutions $\xi, \xi^{'}$ are said to be equivalent if there exists $g\in {\bf G}$, such that $\xi=g\xi^{'}g^{-1}$. We denote by $[{\rm SI}_{\delta,{\bf Spl_G}}({\bf G})]$ the set of equivalence classes in ${\rm SI}_{\delta,{\bf Spl_G}}({\bf G})$.
\end{enumerate}
\end{definition}
Since the choice of quasi-compact form in an inner class only depends on the choice of the pinning of ${\bf G}$, where two pinnings ${\bf Spl_G}$ and ${\bf Spl_G}^{'}$ differ from an inner involution, say, there exists $h \in {\bf G}$, such that ${\bf Spl_G}={\rm int}(h)\circ {\bf Spl_G}^{'}$ (\cite[Proposition~2.8]{ABV91}). Thus $[{\rm SI}_{\delta,{\bf Spl_G}}({\bf G})]$ is independent of the choice of ${\bf Spl_G}$.

We denote by $[{\rm SI}_\delta({\bf G})]$ the set of equivalence classes of strong involutions defined by $\delta$. Note that two equivalent strong involutions $\xi, \xi^{'}$ have the same central invariant. Hence the central invariant is well defined for $[{\rm SI}_\delta({\bf G})]$.

For $\xi \in [{\rm SI}_\delta({\bf G})]$, let $\theta_\xi ={\rm int}(\xi)$ be the corresponding Cartan involution in the inner class defined by $\delta$. In fact, $\xi \mapsto \theta_\xi$ induces a surjection:
$$[{\rm SI}_\delta({\bf G})] \twoheadrightarrow  \{\text{equivalence classes of cartan involutions in the inner class defined by $\delta$} \} . $$
To any Cartan involution $\theta$ of ${\bf G}$ in the inner class defined by $\delta$, one identifies $[\theta]$ with a class in $H^1([\theta_{\rm qc}(\delta)], {\bf G}_{\rm ad})$, by Remark \ref{equivclH1forcartan}, and defines a central invariant ${\rm Inv}([\theta])\in {\bf Z}^{\delta}/(1+\delta){\bf Z}$ of the equivalence class $[\sigma]$,
using the map
\[H^1([\theta_{\rm qc}(\delta)], {\bf G}_{\rm ad})\rightarrow H^2([\theta_{\rm qc}(\delta)], {\bf Z})\cong \widehat{H}^0(\theta, {\bf Z})\cong {\bf Z}^{\theta}/(1+\theta){\bf Z}\cong {\bf Z}^{\delta}/(1+\delta){\bf Z}.\]
The first map is from the connecting homomorphism of group cohomology for the exact sequence:
\[1\rightarrow {\bf Z}\rightarrow {\bf G}\rightarrow {\bf G}_{\rm ad}\rightarrow 1.\] The second and the third arrow are from properties of Tate cohomology, and the last one is from \cite[Lemma 8.6]{AT18}.
If two Cartan involutions $\theta_1, \theta_2$ of ${\bf G}$ live in the same inner class and have the same central invariant, then $H^1(\theta_1, {\bf G})\cong H^1(\theta_2, {\bf G})$ (cf. \cite[Lemma 8.10]{AT18}).

The following commutative diagram shows the compatibility of the two central invariants:
\[\xymatrix{[{\rm SI}_{\delta}({\bf G})]\ar[r]\ar[d]^{\rm Inv} & H^1([\theta_{\rm qc}(\delta)], {\bf G}_{\rm ad}) \ar[d]^{\rm Inv}  \\  {\bf Z}^{\delta}_{\rm tor}\ar[r] &{\bf Z}^{\delta}/(1+\delta){\bf Z}}\]
\begin{proposition}\label{CICartan}\cite[Lemma~6.10]{Ada18}
Suppose that $\theta$ is a Cartan involution of ${\bf G}$ in the inner class defined by $\delta$.
Choose a representative $z\in {\bf Z}^{\delta}_{{\rm tor}}$ of ${\rm Inv}([\theta]) \in {\bf Z}^{\delta}/(1+\delta) {\bf Z}$\footnote{For existence of such representative, see \cite[Lemma~2.15]{ABV91}.}. Then there is a bijection:
\[ H^1(\theta,{\bf G})\leftrightarrow\{\text{equivalent classes of strong involutions with central invariant } z \}.\]
\end{proposition}
\subsubsection{Strong Real Forms and their relations}

\begin{definition}\cite[Definition 8.11]{AT18}
\begin{enumerate}
\item A {\it strong real form} in the inner class of $\sigma_{\rm qs}(\delta, {\bf Spl_G})$ is an element of 
\[{\rm SRF}_{\sigma_{\rm qs}(\delta, {\bf Spl_G})}({\bf G}):=Z^1(\sigma_{\rm qs}(\delta, {\bf Spl_G}), {\bf G}; {\bf Z}_{\rm tor})/(1+\sigma_{\rm qs}(\delta, {\bf Spl_G})){\bf Z},\]
where $Z^1(\sigma_{\rm qs}(\delta, {\bf Spl_G}), {\bf G}; {\bf Z}_{\rm tor}):=\{g\in {\bf G}: g\sigma_{\rm qs}(\delta, {\bf Spl_G})(g)\in {\bf Z}_{\rm tor}\}$.
Moreover, to a strong real form $g\in {\rm SRF}_{\sigma_{\rm qs}(\delta, {\bf Spl_G})}({\bf G})$, we can associate a central invariant
\[{\rm Inv}(g)=g\sigma_{\rm qs}(\delta, {\bf Spl_G})(g)\in {\bf Z}_{\rm tor}^{\delta}.\]

\item Two strong real forms $g, h$ are said to be equivalent if they map to a same element of $H^1(\sigma_{\rm qs}(\delta, {\bf Spl_G}), {\bf G}; {\bf Z}_{\rm tor}):=Z^1(\sigma_{\rm qs}(\delta, {\bf Spl_G}), {\bf G}; {\bf Z}_{\rm tor})/[g\sim t g\sigma_{\rm qs}(\delta, {\bf Spl_G})(t^{-1}), t\in {\bf G}]$.
\end{enumerate}
\end{definition}
Note that two equivalent strong real forms $g, h$ have the same central invariant. Note that for two pinnings ${\bf Spl_G}$ and ${\bf Spl_G}^{'}$ with 
$h \in {\bf G}$, such that ${\bf Spl_G}={\rm int}(h)\circ {\bf Spl_G}^{'}$,we have a bijection:
$${\rm SRF}_{\sigma_{\rm qs}(\delta, {\bf Spl_G})}({\bf G})\cong {\rm SRF}_{\sigma_{\rm qs}(\delta, {\bf Spl_G}^{'})}({\bf G}),\quad g\mapsto gh \sigma_{\rm qs(\delta,{\bf Spl_G}^{'})}(h)^{-1}.$$

Thus we define the set $[{\rm SRF}_{\delta}({\bf G})]$ of equivalence classes of strong real forms in the inner class defined by $\delta$, together with a central invariant map ${\rm Inv}: [{\rm SRF}_{\delta}({\bf G})]\rightarrow Z_{\rm tor}^{\delta}$. By the previous discussion, $[{\rm SRF}_{\delta}({\bf G})]$ is independent of the choice of ${\bf Spl_G}$.

As the pinning varies, maps $g\in {\rm SRF}_{\sigma_{\rm qs}(\delta, {\bf Spl_G})}\mapsto {\rm int}(g)\circ \sigma_{\rm qs}(\delta, {\bf Spl_G})$ are compatible and induce a surjection:
$$[{\rm SRF}_{\delta}({\bf G})]\twoheadrightarrow  \{\text{equivalence classes of real forms in the inner class defined by $\delta$} \} . $$

To any real form $\sigma$ of ${\bf G}$ in the inner class defined by $\delta$, one identifies $[\sigma]$ with a class in $H^1([\sigma_{\rm qs}(\delta)], {\bf G}_{\rm ad})$, by Remark \ref{equivclH1}, and defines a central invariant ${\rm Inv}([\sigma])\in {\bf Z}^{\delta}/(1+\delta){\bf Z}$ of the equivalence class $[\sigma]$,
using the map
\[H^1([\sigma_{\rm qs}(\delta)], {\bf G}_{\rm ad})\rightarrow H^2(\sigma_{\rm qs}(\delta), {\bf Z})\cong \widehat{H}^0(\sigma, {\bf Z})\cong {\bf Z}^{\sigma}/(1+\sigma){\bf Z}\cong {\bf Z}^{\delta}/(1+\delta){\bf Z}\]
which is similar to the Cartan involution case.

The following commutative diagram shows the compatibility of the two central invariants:
\[\xymatrix{[{\rm SRF}_{\delta}({\bf G})]\ar[r]\ar[d]^{\rm Inv} & H^1([\sigma_{\rm qs}(\delta)], {\bf G}_{\rm ad}) \ar[d]^{\rm Inv}  \\  {\bf Z}^{\delta}_{\rm tor}\ar[r] &{\bf Z}^{\delta}/(1+\delta){\bf Z}}\]
\begin{proposition}\label{CIReal}\cite[Proposition~8.14]{AT18}
Suppose that $\sigma$ is a real form of ${\bf G}$ in the inner class defined by $\delta$.
Choose a representative $z\in {\bf Z}^{\delta}_{{\rm tor}}$ of ${\rm Inv}([\sigma]) \in {\bf Z}^{\delta}/(1+\delta) {\bf Z}$. Then there is a bijection:
\[H^1(\sigma,{\bf G})\leftrightarrow \{\text{equivalent classes of strong real forms with central invariant } z \}.\]
\end{proposition}
Recall that we have the canonical bijection of pointed sets ${H^1(\sigma,{\bf G}_{\rm ad})}\cong {H^1(\theta,{\bf G}_{\rm ad})}$ (see Remark \ref{equivclH1forcartan}), and two compatibilities of central invariant functions mentioned in this section. We have the following commutative diagram by composing them:
\[\begin{tikzcd}
	{[{\rm SRF}_{\delta}({\bf G})]} & {H^1([\sigma], {\bf G}_{\rm ad})} & {H^1([\theta], {\bf G}_{\rm ad})} & {[{\rm SI}_\delta({\bf G})]} \\
	{{\bf Z}^{\delta}_{\rm tor}} & {{\bf Z}^{\delta}/(1+\delta){\bf Z}} & {{\bf Z}^{\delta}/(1+\delta){\bf Z}} & {{\bf Z}^{\delta}_{\rm tor}}
	\arrow[from=1-1, to=1-2]
	\arrow["{{\rm Inv}}", from=1-1, to=2-1]
	\arrow["{\cong}", from=1-2, to=1-3]
	\arrow["{{\rm Inv}}", from=1-2, to=2-2]
	\arrow["{{\rm Inv}}", from=1-3, to=2-3]
	\arrow[from=1-4, to=1-3]
	\arrow["{{\rm Inv}}", from=1-4, to=2-4]
	\arrow[from=2-1, to=2-2]
	\arrow["{=}", from=2-2, to=2-3]
	\arrow[from=2-4, to=2-3]
\end{tikzcd}\]
This shows the compatibility of the central invariants for strong involutions and strong real forms (in ${\bf Z}^{\delta}/(1+\delta){\bf Z}$).

Note that there is a natural bijection between $[{\rm SRF}_{\delta}({\bf G})]$ and $[\mathrm{SI}_{\delta}({\bf G})]$ (\cite[Remark~8.13]{AT18}): fixing a pinning ${\bf Spl_G}$, choose $g_0\in {\rm SRF}_{\sigma_{\rm qs}(\delta, {\bf Spl_G})}$ such that $\sigma_{\rm qc}(\delta,{\bf Spl_G})=\mathrm{int}(g_0) \circ {\sigma_{\rm qs}(\delta, {\bf Spl_G})}$. Then by twisting $g_0$ (see \cite[Lemma~2.4]{AT18} and \cite[Corollary~4.7]{AT18}), we have
$$[{\rm SRF}_{\delta}({\bf G})]\cong H^1([\sigma_{\rm qc}],{\bf G};{\bf Z}_{\rm tor})\cong H^1([\theta_{\rm qc}],{\bf G};{\bf Z}_{\rm tor})\cong [\mathrm{SI}_{\delta}({\bf G})].$$

We conclude that the definitions of strong involutions and strong real forms are compatible, also their central invariants (in ${\bf Z}^{\delta}/(1+\delta){\bf Z}$). From now on, we do not differ $[\mathrm{SI}_{\delta}({\bf G})]$ ``the equivalence classes of strong involutions (in the inner class defined by $\delta$)'' and $[{\rm SRF}_{\delta}({\bf G})]$ ``the equivalence classes of strong real forms (in the inner class defined by $\delta$)''.

For the equal rank case $\delta=\theta_{\rm qc}=1$, we have an explicit description of equivalence classes of strong real forms using central invariants:
\begin{proposition}\cite[Proposition~8.16]{AT18}\cite[Corollary~6.14]{Ada18}\label{srf1}
Suppose that $\sigma$ is an equal rank real form of ${\bf G}$.
Choose $x\in {\bf G}$ such that ${\rm int}(x)$ is a Cartan involution for $\sigma$ and $z=x^2\in {\bf Z}$. Then we have an explicit bijection
\[H^1(\sigma,{\bf G})\leftrightarrow  S(z),\]
where $S(z)$ is the set of conjugacy classes of ${\bf G}$ with square equal to $z$. If ${\bf H}$ is a Cartan subgroup of ${\bf G}$ with Weyl group $W$, then $S(z)$ is equal to $\{h\in {\bf H}: h^2=z\}/W$.
\end{proposition}

\begin{example}\label{exampleCnDn}
In \cite[Table~1:~Classical~groups]{AT18} and \cite[Table~4:~Adjoint~classical~groups]{AT18}, the cardinality of the cohomology group $H^1(\sigma,{\bf G})$ and $H^{1}(\sigma,{\bf G}_{\rm ad})$ (of equal rank case) are calculated for the classical groups $\bf{G}$ respectively. In this example, we will explain how this parametrization works for the symplectic group and orthogonal group, following the calculation in \cite[Example~5.11]{AC09}. Recall that in the equal rank case, we choose $\delta=\theta_{\rm qc}=1.$

To describe the set $\{h\in {\bf H}: h^2 \in {\bf Z}\}/W$, we use the coweight lattice\footnote{This is actually a lattice only if $\bf G$ is semisimple. } for $\bf G$:
\[ P^\vee=\{ \lambda^\vee \in X^{\vee}({\bf H}) \otimes _{\bb Z}{\bb C}| \langle\alpha, \lambda^{\vee} \rangle \in {\bb Z} \text{ for all }\alpha \in \Phi \},\]
which can be described explicitly via the given based root datum. In fact,
through the canonical isomorphism ${\mathfrak{h} \cong }X^{\vee}({\bf H})\otimes_{\bb Z}{\bb C}$, the coweight lattice $P^\vee$ for $\bf G$ can be regarded as a subset of Lie algebra $\mathfrak{h}$:
\[P^\vee=\{ \lambda ^\vee \in \mathfrak{h} |\exp (2\pi i\lambda ^\vee)\in \Z\}.\]
As a result, by \cite[Lemma~22.3]{ABV91}, the set  $\{h\in {\bf H}: h^2 \in {\bf Z}\}/W$ can be identified with the set $(P^{\vee}/2X^{\vee}({\bf H}))/W$, whose inverse is induced by the map $\lambda^\vee\in P^{\vee}\mapsto {\rm exp}(\pi i\lambda^\vee)\in {\bf H}$.

{\it Symplectic Case $C_n$:} For ${\bf G}={\rm Sp}_{2n}({\bb C})$, since it is simply connected and semisimple, we fix an isomorphism $X^\vee({\bf H})=\Phi^{\vee}\cong {\bb Z}^n$ through the root system chosen in \ref{HCSp}. By Proposition \ref{srf1}, the equivalence classes of strong real forms are parametrized by
\[([{\bb Z}^n \cup ({\bb Z}+\frac{1}{2})^n]/2{\bb Z}^n)/W.\]
For representatives, we choose $\frac{1}{2}(1,...,1)$ and $(\overbrace{1,...,1}^{p},\overbrace{0,...,0}^{q})$, with $0\leq p \leq n.$
The representative $\frac{1}{2}(1,...,1)$ corresponds to the split real symplectic group ${\rm Sp}_{2n}({\bb R})$ and the cohomology set $\left\lvert H^1(\sigma_{{\rm Sp}_{2n}({\bb R})},{\bf G}) \right\rvert=1$. The central invariant of $\frac{1}{2}(1,...,1)$ is $-I$. Representatives $(\overbrace{1,...,1}^{p},\overbrace{0,...,0}^{q})$, with $0\leq p \leq n$, correspond to the quaternionic symplectic group\footnote{the isometry group of a Hermitian form on a quaternionic vector space.} ${\rm Sp}(p,q),$ with $p+q=n$, and the cohomology set $\left\lvert H^1(\sigma_{{\rm Sp}(p,q)},{\bf G} )\right\rvert=p+q+1$. The central invariant of $(\overbrace{1,...,1}^{p},\overbrace{0,...,0}^{q})$, with $0\leq p \leq n$, is $I$. All the equivalence classes of strong real forms of ${\rm Sp}_{2n}$ are listed above.

{\it Even Orthogonal Case $D_n$:} For $G={\rm SO}(p,q)$ with both $p$ and $q$ even, by \cite[Example~8.20]{AT18}, the set $S(1)=\{h\in {\bf H}: h^2 =1\}/W$ is the same as the set 
\[\{{\rm diag}(I_r,-I_s)|r+s=p+q;s\ {\rm even}\}.\] Hence, the set $H^1(\sigma_G,{\bf G})$ has cardinality $\frac{p}{2}+\frac{q}{2}+1$, and it parametrizes the even special orthogonal group ${\rm SO}(r,s)$, with $r+s=2n$ and both $r$ and $s$ even. Such groups have the central invariant $I$.
\begin{remark}
The isometry group ${\bf SO}^*(2n)$ of a skew-Hermitian form on a quaternionic vector space (counted twice) is the remaining equivalence class of strong real forms of ${\bf SO}(2n,\C)$. The central invariant of this group is $-I$. This corresponds to $H^1(\sigma_{{\rm SO}^*(2n)},{\bf SO}(2n,\C))$ with cardinality $2$.
\end{remark}
\end{example}
\begin{remark}
When $p+q=2n$ with $n \geq 5$, we have ${\rm Out}({\bf SO}(2n,{\bb C}))\cong {\bf Z}/2{\bf Z}$. Thus there exists non-equal rank real form of ${\rm SO}(2n,\C)$: for ${\rm SO}(p,q)$ with both $p,q$ odd, note that ${\rm Rank}({\rm SO}(p,q))=\frac{p+q}{2}>\frac{p-1}{2}+\frac{q-1}{2}={\rm Rank}({\rm SO}(p)\times {\rm SO}(q))$. The inner class of non-equal rank real form consists of ${\rm SO}(p,q)$ with both $p,q$ odd. This explains why only ${\rm SO}(p,q)$, with both $p$ and $q$ even, appeared in Example \ref{exampleCnDn}.
\end{remark}

\subsubsection{Representation of strong involutions}
In this subsection, let $G$ be a quasi-split real reductive group of equal rank, and let $G^d$ be its complex dual group. We denote by ${\bf G}$ the complexification of $G$ and by $\sigma_G$ the action of $\sigma$ on ${\bf G}$ associated to $G$. We also fix a Borel pair $(\mathbf{B},\mathbf{T})$ of $\mathbf{G}$. Let $W=\mathrm{Norm}_{\mathbf{G}}(\mathbf{T})/\mathbf{T}$ be the Weyl group associated to $(\mathbf{B},\mathbf{T})$. We denote by $\rho$ (resp. $\rho^{\vee}$) the half sum of the positive roots (resp. coroots) associated to $(\mathbf{B},\mathbf{T})$.

For simplicity, we denote by $\tilde{S}$ as the set of strong involutions of ${\bf G}$ in the inner class defined by $\delta=\theta_{\rm qc}=1$, and $[\tilde{S}]$ as the set of equivalent classes of such strong involutions. We have a stratification $\tilde{S}= \cup_{z\in  {\bf Z}}\tilde{S}(z)$ with $\tilde{S}(z)=\{\tilde{x}\in {\bf T}: \tilde{x}^2=z\}$, called the set of strong involutions of type $z$. By proposition \ref{srf1}, we identify $[\tilde{S}](z)=\tilde{S}(z)/W$ with the set of conjugacy classes of strong real forms with central invariant $z$. For a strong involution $\tilde{x}\in \tilde{S}$, set $\theta_{\tilde{x}}={\rm int}(\tilde{x})$ and ${\bf K}_{\tilde{x}}={\bf G}^{\theta_{\tilde{x}}}={\rm Cent}_{\bf G}(\tilde{x})$. 
\begin{definition}\begin{enumerate}
 \item[(1)]   A representation of a strong involution $\tilde{x}\in \tilde{S}$ is a pair $(\tilde{x}, \pi)$ where $\pi$ is a $(\mathfrak{g}, {\bf K}_{\tilde{x}})$-module.
\item[(2)]Two representations $(\tilde{x}, \pi)$ and $(\tilde{x}^{'},\pi^{'})$ of strong involutions are equivalent if there exists $g\in {\bf G}$ such that $g\tilde{x}g^{-1}=\tilde{x}^{'}$ and $\pi^{'}\cong \pi^g$.
\item[(3)] A representation of a strong real form associated to $x\in [\tilde{S}](z)$ is an equivalence class $[\tilde{x}, \pi]$ of a representation $(\tilde{x}, \pi)$ of a strong involution $\tilde{x}\in \tilde{S}(z)$ lifting $x$.
\end{enumerate}
\end{definition}

For any strong real form $x\in [\tilde{S}](z)$ with lifting $\tilde{x}\in \tilde{S}(z)$, two representations $[\tilde{x}, \pi]$ and $[\tilde{x}, \pi^{'}]$ of a strong real form $x$ are the same (i.e. $[\tilde{x}, \pi]=[\tilde{x}, \pi^{'}]$) if and only if there exists $g\in {\bf G}$ such that $g\tilde{x}g^{-1}=\tilde{x}$ and $\pi^g\cong\pi^{'}$. This holds if and only if $g\in {\bf K}_{\tilde{x}}$.
 Moreover, for any regular $\lambda\in \mathfrak{t}^*$, denote by $\pi_{\tilde{x}}(\lambda)$ the unique $(\mathfrak{g},\mathbf{K}_{\tilde{x}})$-module (up to isomorphism) with parameter $\lambda$.
Then for the representation $[\tilde{x},\pi_{\tilde{x}}(\lambda)]$ of strong real form $x$ and for $w\in W$, we have 
\begin{equation}\label{reecrit}
    [\tilde{x},\pi_{\tilde{x}}(\lambda)]=w [\tilde{x},\pi_{\tilde{x}}(\lambda)]=[w\tilde{x},\pi_{w\tilde{x}}(w\lambda)].
\end{equation}

\subsubsection{\textit{L}-packet associated to discrete Langlands parameter and strong involution}\label{VoganLpacket}

\begin{definition}
 If a L-parameter $\varphi$ of ${\bf G}$ has a finite centralizer, then $\varphi$ is called a discrete series L-parameter.
\end{definition}

Fix a discrete series L-parameter $\varphi$ of $\mathbf{G}$. Recall that $W_{\bb R}\cong \mathbb{C}^\times\cup j \mathbb{C}^\times$. As explained in \cite[Proposition 2.10]{AV92}, we can associate $\varphi$ with a regular character $\lambda\in\mathfrak{t}^*$. More precisely, after conjugating by $\mathbf{G}^d$ we may assume $\varphi(\mathbb{C}^\times)\subset \mathbf{T}^\vee$. The discrete property implies that 
\begin{enumerate}
\item[(1)] $\varphi(j)h\varphi(j)^{-1}=h^{-1}$, for $h  \in {\bf T}^\vee$, 
\item[(2)] for $\mathfrak{ z}\in {\bb C}^{\times}$, we have $\varphi(\mathfrak{ z})=(\mathfrak{ z}/\bar{\mathfrak{ z}})^{\lambda}$, with $\lambda\in X(\bf T)\otimes_{\bb Z}{\bb C}$ regular. 
\end{enumerate}
By the canonical isomorphism $X(\mathbf{T})\otimes_{\bb Z}\mathbb{C}\cong \mathfrak{t}^*$, this regular $\lambda\in X(\mathbf{T})\otimes\mathbb{C}$ can associate to a character of the Lie algebra $\mathfrak{t}$, still denoted by $\lambda$. By \cite[\S 7]{AC09}, we may assume $\lambda$ is dominant with respect to the root system associated with $(\mathbf{B},\mathbf{T})$. Let $\mu$ be the infinitesimal character associated to $\lambda$, which is independant of choices of conjugations by ${\bf G}^d$.

Let $\tilde{x}\in \tilde{S}(z)$ be a strong involution with $z\in \Z$. We denote by $\Pi(\tilde{x}, \varphi)$ the L-packet associated with $\varphi$ and the strong involution $\tilde{x}$ (i.e. the L-packet of $\varphi$ and the real form $G_{\tilde{x}}$ of ${\bf G}$ determined by $\tilde{x}$). By the equivalence between the category of smooth Fr\'echet representations of moderate growth of $G_{\tilde{x}}$ and the category of admissible finitely generated $(\mathfrak{g}, {\bf K}_{\tilde{x}})$-modules, $\Pi(\tilde{x}, \varphi)$ can be identified with the set of discret series $(\mathfrak{g}, {\bf K}_{\tilde{x}})$-modules with infinitesmal character $\mu$. Note that for any regular character $\lambda\in {\mathfrak t}^{*}$, $\pi_{\tilde{x}}(\omega\lambda)=\pi_{\tilde{x}}(\lambda)$ if and only if $w\in {\bf K}_{\tilde{x}}$. The L-packet $\Pi(\tilde{x}, \varphi)$ can be parametrized as follows
\[\Pi(\tilde{x}, \varphi)=\{\pi_{\tilde{x}}(\omega^{-1}\lambda): \omega\in W/W({\bf T}, {\bf K}_{\tilde{x}})\},\]
where $W({\bf T}, {\bf K}_{\tilde{x}})$ is the Weyl group of ${\bf T}$ in ${\bf K}_{\tilde{x}}$.

Vogan defined ${\Pi}(\varphi)$ to be the set of all the representations of strong real forms with infinitesimal character $\mu$. Since $[\tilde{S}](z)=\tilde{S}(z)/W$ can be identified with the set of conjugacy classes of strong real forms with central invarint $z$, the set of representations of a strong real form with central invariant $z$ is stable under the action of $W$. The L-packet $\Pi(\tilde{x}, \varphi)$ can be embedded into ${\Pi}(\varphi)$ as
\[\Pi(\tilde{x}, \varphi)=\{[\tilde{x},\pi_{\tilde{x}}(w^{-1}\lambda)]: w\in W/W({\bf T}, {\bf K}_{\tilde{x}})\},\]
which induces a $W$-equivariant bijection
$\Pi(\tilde{x},\varphi)\leftrightarrow \{w\tilde{x}\vert w\in W/W({\bf T}, {\bf K}_{\tilde{x}}) \}$ 
using the equality (\ref{reecrit}).
To make the bijection more explicit, choose a set $S_0$ of representatives of $[\tilde{S}]$. Thus, we have a $W$-equivariant bijection
\[\Pi(\varphi)=\coprod _{\tilde{x}\in S_0} \Pi(\tilde{x}, \varphi)\leftrightarrow \coprod _{\tilde{x}\in S_0} \{w\tilde{x}\vert w\in W/W({\bf T}, {\bf K}_{\tilde{x}}) \}\cong \tilde{S}\]
Thus we obtain a $W$-equivariant bijection.
\begin{equation}\label{tildeSPiphi}
    \tilde{S} \leftrightarrow \Pi(\varphi),\quad \tilde{x} \mapsto [\tilde{x},\pi_{\tilde{x}}(\lambda)]. 
\end{equation}
We regroup Vogan's L-packet of $\varphi$ for inner forms using the central invariant: for $z\in \Z$, we set
\[\Pi_z(\varphi):=\cup_{\tilde{x}\in \tilde{S}(z)} \Pi(\tilde{x}, \varphi).\]
\begin{proposition}\cite[Prop. 5.3]{Ada11}\label{Ada115}
    For any $z\in \Z$ and any discrete series Langlands parameter $\varphi$ of ${\bf G}$, we have a $W$-equivariant bijection between $\tilde{S}(z)$ and $\Pi_z(\varphi)$, given by $\tilde{x}\mapsto [\tilde{x}, \pi_{\tilde{x}}(\lambda)]$, where $\lambda\in {\mathfrak t}^{*}$ is the regular character determined by $\varphi$. 
\end{proposition}
Let $\tilde{x}\in \tilde{S}$ and $\theta_{\tilde{x}}$ be the Cartan involution associated to $\tilde{x}$. By \cite[Theorem~6.2(f)]{Vog78}, a discrete series representation $\pi(\lambda)$ is generic if and only if every simple root of $\mathfrak{g}$ in the chamber defined by $\lambda$ is non-compact. Recall that a root $\alpha$ is compact with respect to the Cartan involution $\theta_{\tilde{x}}$ if $\alpha(\tilde{x})=1$, and is non-compact if $\alpha(\tilde{x})=-1$. Thus, the element $\tilde{x} \in \tilde{S}$ corresponds to a generic discrete series representation through the bijection (\ref{tildeSPiphi}) if and only if $\alpha(\tilde{x})=-1$ for all simple roots $\alpha$.
\subsection{Langlands-Vogan parameters}\label{subsec2.6}
In this section, we will give the explicit description of the component group of a discrete series L-parameter for symplectic groups and orthogonal groups starting from a given based root datum, and explain how to associate a Langlands-Vogan parameter with a Harish-Chandra parameter.

Fix a fundamental Borel pair $({\bf B}_{*}, {\bf T}_{*})$ of Whittaker type of ${\bf G}$. Let $D_b=(X,\Phi,\Delta, X^\vee, \Phi^\vee, \Delta^\vee)$ be the based root datum associated to $({\bf B}_{*}, {\bf T}_{*})$. We denote by $\Psi_{*}$ the set of positive roots generated by $\Delta$.  Set $\rho=\frac{1}{2}\sum\limits_{\alpha\in \Psi_{*}}\alpha$ and $\rho^{\vee}=\frac{1}{2}\sum\limits_{\alpha^{\vee}\in \Psi_{*}^{\vee}}\alpha^{\vee}$. The half sum $\rho$ of the positive roots produces a basepoint $x_b\in \tilde{S}$ of the strong involutions of ${\bf G}$:
\[x_b=\exp(i\pi \rho^{\vee})\in \tilde{S}=\cup_{z\in  {\bf Z}} \tilde{S}(z)\subset {\bf T}_{*}.\]
Note for any simple root $\alpha\in \Delta$, we have $\langle \alpha, \rho^{\vee}\rangle=1$.
We can deduce that $\alpha(x_b)=\exp(i\pi\langle \alpha, \rho^{\vee}\rangle)=-1$, for any simple root $\alpha\in \Delta$. Thus, the element $x_b \in \tilde{S}$ corresponds to a generic discrete series representation through the bijection (\ref{tildeSPiphi}) for any discrete series L-parameter $\varphi$ of ${\bf G}$.

\begin{example}\label{basepoint}
 We use the fixed root systems $\Psi_{b, \rm Sp}$ in section \ref{HCSp} and $\Psi_{b, \rm O}$ in section \ref{HCO} to compute the basepoints for ${\bf G}$ following the above discussion. 
    \begin{enumerate}
       
        \item Suppose $\mathbf{G}^d=\mathrm{O}(2n+1,\mathbb{C})$ with $n$ odd. Let $\mathbf{T}$ be the maximal torus of $\mathbf{G}^d$ with Lie algebra  \[ \mathfrak{t}=\{{\rm diag}(g(t_1),\cdots, g(t_{n}),0): t_i\in {\bb C}\}\]
with $g(t)=(\begin{smallmatrix}0& t\\ -t& 0\end{smallmatrix})$ for all $t\in {\bb C}$. Let \[e_i:(\mathrm{diag}(g(t_1),\cdots,g(t_n),0))\mapsto 2it_i.\] Then $\mathfrak{t}^*=\oplus_{i=1}^n\mathbb{C}e_i$. There is a positive root system $$\Psi^{\vee}=\langle e_1+e_{n}, -e_{n}-e_{2},e_2+e_{n-1},\cdots,-e_{\frac{n+3}{2}}-e_{\frac{n+1}{2}},e_{\frac{n+1}{2}}\rangle.$$ The half sum of these positive roots is $$\rho^{\vee}=\frac{1}{2}((2n-1)e_1+(2n-5)e_2+\cdots +3e_{\frac{n+1}{2}}-e_{\frac{n+3}{2}}-\cdots -(2n-3)e_n).$$ Thus, the basepoint of strong involutions for ${\bf G}$ is \[x_{b,\rm Sp}=(i\sin(\frac{2n-1}{2}\pi),\cdots,i\sin(\frac{2n-3}{2}\pi))=i(1,\cdots,1,-1,\cdots,-1)\] and $z(\rho^{\vee})=x_{b,\mathrm{Sp}}^2=-I$.

For  $\mathbf{G}^d=\mathrm{O}(2n+1,\mathbb{C})$ with $n$ even, the same computation shows the basepoint for $\mathbf{G}$ is \[x_{b,\mathrm{Sp}}=i(-1,\cdots,-1,1,\cdots,1).\]

\item Suppose $\mathbf{G}^d=\mathrm{O}(2n,\mathbb{C})$ with $n$ even. Then as in section $\ref{HCO}$, $$\Psi^{\vee}=\langle e_i-f_i,f_i-e_{i+1},e_{\frac{n}{2}}\pm f_{\frac{n}{2}} \rangle_{1\leq i\leq \frac{n}{2}-1}$$ is a positive root system. The half sum of these positive roots is\[\rho^{\vee}=\frac{1}{2}(2(n-1)e_1+2(n-3)e_2+\cdots + 2e_{\frac{n}{2}} + 2(n-2)f_1 + \cdots +4f_{\frac{n}{2}-1}).\]  Thus, the basepoint of strong involutions for ${\bf G}$ is \[x_{b, \rm O}=(\cos((n-1)\pi),\cdots,\cos (0))=((-1)^{n-1},\cdots,(-1)^{n-1}, (-1)^n, \cdots,(-1)^n)\] and $z(\rho^{\vee})=x_{b,\mathrm{O}}^2=I$. 
    \end{enumerate}
\end{example}

Let $\varphi$ be an L-parameter associated to a limit of discrete series representation of an equal rank real form $G$ of ${\bf G}$.
We describe the component group $A_{\varphi}=\pi_0(C_{\varphi})$ for $G={\rm Sp}_{2n}({\bb R})$ and $G={\rm O}(p,q)$ with $p,q$ even in the following example.
\begin{example}\label{Lparameterdecom}
Let ${\bf 1}$ be the trivial character of $W_{\bb R}$.

(1) Let $G={\rm Sp}_{2n}({\bb R})$. Let us write the limit of discrete series $L$-parameter $\varphi$ of $G$ as
\begin{equation}\label{parasym}
\varphi=\bigoplus_{i=1}^kc_i\rho_{\lambda_i}\bigoplus (2z+1) {\bf 1},
\end{equation}
where 
 $z$ is a positive integer and 
 $\rho_{\lambda_1}, \cdots, \rho_{\lambda_k}$ are self-dual irreducible representations of $W_{\bb R}$ of dimension $2$ with $\lambda_i$ even natural number and $c_i>0$ an integer.

The component group $A_{\varphi}$ of $\varphi$ is 
\[A_{\varphi}=\begin{cases}\oplus_{i=1}^k ({\bb Z}/2{\bb Z})a_i, &\text{ if } z=0,\\
\oplus_{i=1}^k ({\bb Z}/2{\bb Z})a_i \oplus ({\bb Z}/2{\bb Z}) b, & \text{ if } z>0,\end{cases}\] where $a_i$ is a symbol corresponding to $\rho_{\lambda_i}$ and $b$ is a symbol corresponding to ${\bf 1}$. 

(2) Let $G={\rm O}(p,q)$ with $p,q$ even. Let $\varphi$ be a limit of discrete series $L$-parameter with decomposition 
\begin{equation}
\varphi=\bigoplus_{i=1}^kc_i\rho_{\lambda_i}\bigoplus 2z{\bf 1},
\end{equation}
where $z$ is a positive integer and $\rho_{\lambda_1}, \cdots, \rho_{\lambda_k}$ are self-dual irreducible representations of $W_{\bb R}$ of dimension $2$ with $\lambda_i$ even natural number and $c_i\in{\bb N}_{>0}$.

The component group $A_{\varphi}$ of $\varphi$ is 
\[A_{\varphi}=\begin{cases}\oplus_{i=1}^k ({\bb Z}/2{\bb Z})a_i, &\text{ if } z=0,\\
\oplus_{i=1}^k ({\bb Z}/2{\bb Z})a_i \oplus ({\bb Z}/2{\bb Z}) b, & \text{ if } z>0,\end{cases}\] where $a_i$ is a symbol corresponding to $\rho_{\lambda_i}$ and $b$ is a symbol corresponding to ${\bf 1}$. 
\end{example}
\begin{definition}
    A Langlands-Vogan parameter of $G$ is a pair $(\varphi, \eta)$, where $\varphi$ is a $L$-parameter of $G$ and $\eta$ is a character of the component group ${A}_{\varphi}$.
\end{definition}

If $\varphi$ is a discrete series L-parameter of $G$, we have $$A_{\varphi}\cong\{h\in {\bf T}_{*}^{\vee}: h^2=1\}=\mathbf{T}_*^\vee[2]$$ and the pair $(\varphi,\eta)$ is called a discrete series Langlands-Vogan parameter. Using the perfect pairing between $\mathbf{T}_*^\vee$ and $X(\mathbf{T}_*^\vee)$, we identify $\mathbf{T}_{*}^{\vee}[2]$ with $${\rm Hom}(X({\bf T}_{*}^{\vee})/2X({\bf T}_{*}^{\vee}), {\bb C}^{\times})={\rm Hom}(X^{\vee}({\bf T}_{*})/2X^{\vee}({\bf T}_{*}), {\bb C}^{\times}).$$ This induces an isomorphism $$s:X^{\vee}({\bf T}_{*})/2X^{\vee}({\bf T}_{*})\cong \widehat{A_{\varphi}}.$$
On the other hand, for any $x\in\tilde{S}(z(\rho^\vee))$, there exists a $\gamma^\vee\in X^\vee(\mathbf{T}_*)$ such that\footnote{We use our fixed isomorphism $\mathfrak{t_*}\cong X^\vee(\mathbf{T}_*)$.} $x=\exp(\pi i \gamma^\vee)$. Note that we have an isomorphism $\tilde{S}(z(\rho^\vee))\cong \mathbf{T}_*[2],x\mapsto xx_b^{-1}$, where $x_b$ is the basepoint of the strong involutions of $\mathbf{G}$. Then there is an isomorphism
\[\tilde{S}(z(\rho^{\vee}))\xrightarrow{\sim} \mathbf{T}_{*}[2]\xrightarrow{\sim}\widehat{A_{\varphi}},\quad x\mapsto xx_b^{-1}\mapsto s( \langle -, xx_b^{-1}\rangle),\]
where $\langle, \rangle$ is the perfect pairing between ${\bf T}_{*}^{\vee}[2]$ and  ${\bf T}_{*}[2]$. In particular, the identity of $\widehat{A_\varphi}$ is the image of the basepoint $x_b\in \tilde{S}(z(\rho^{\vee}))\subset{\bf T}_{*}$ of the strong involutions of ${\bf G}$.

For any $\eta\in\widehat{A_{\varphi}}$, we can associate a unique $x_{\eta}\in\tilde{S}(z(\rho^{\vee}))$. Thus there is a unique element $w_\eta\in W$ such that $w_\eta\cdot x_b=x_\eta$. This implies the element $\eta$ corresponds to a representation $[x_\eta, \pi_{x_\eta}(\lambda)]$ of strong real form defined by $x_\eta$.
Note that, $w_\eta=1$ corresponds to the generic representation $[x_b, \pi_{x_b}(\lambda)]$ with respect to the fundamental Borel pair of Whittaker type $({\bf B}_{*},{\bf T}_{*})$. Thus, for a discrete series Langlands-Vogan parameter $(\varphi, \eta)$, we get a Harish-Chandra parameter $(\lambda_\eta, \Psi_\eta,\mu_\eta)$, where $\Psi_\eta=w_\eta\Psi_{*}$.

\begin{example}\label{BaseSp}
In this example, we choose the basepoint $x_{b,\rm Sp}$ and $x_{b,\rm O}$ described in Example \ref{basepoint}.
\begin{enumerate}
    \item[(1)] Suppose $G=\mathrm{Sp}_{2n}(\mathbb{R})$.
\begin{enumerate}
    \item[(2)] If $n$ is even, then the positive root system $\Psi_{b,\rm Sp}$ is generated by the simple roots
$ \{ e_1+e_n,-e_n-e_2,\cdots,e_{\frac{n}{2}}+e_{\frac{n+2}{2}},-2e_{\frac{n+2}{2}}  \}$. By the dominant condition, the corresponding lifting of $\lambda_0$ is
    \[\lambda_{d, \rm Sp}=(\lambda_1,\lambda_3,\cdots,\lambda_{n-1},-\lambda_n,\cdots,-\lambda_2).\] 
    \item If $n$ is odd, then the corresponding positive root system $\Psi_{b,\rm Sp}$ is generated by the simple roots $\{ e_1+e_n,-e_n-e_2,\cdots,-e_{\frac{n+1}{2}}-e_{\frac{n+3}{2}},2e_{\frac{n+1}{2}}  \}$. By the dominant condition, the corresponding lifting of $\lambda_0$ is
    \[\lambda_{d, \rm Sp}=(\lambda_1,\lambda_3,\cdots,\lambda_n,-\lambda_{n-1},\cdots,-\lambda_2).\] 
\end{enumerate}

In both cases, the unique generic representation in the L-packet corresponding to the basepoint is determined by the Harish-Chandra parameter $(\lambda_{d,\rm Sp}, \Psi_{b, \rm Sp})$.

   \item[(2)] Suppose $G$ is the even orthogonal group of rank $n$ associated to the basepoint $x_{b, \rm O}$. Thus $G$ has signature $(n, n)$ (resp. $(n+1, n-1)$) if $n$ is even (resp. odd). Let $\lambda_0$ be an infinitesimal character of $G$ as above. 
\begin{enumerate}
    \item If $n$ is even, then the corresponding positive root system $\Psi_{b, \rm O}$ is generated by the simple roots $\{ e_1-f_1,f_1-e_2,\cdots, e_{\frac{n}{2}}-f_{\frac{n}{2}}, e_{\frac{n}{2}}+f_{\frac{n}{2}} \}$. By the dominant condition, the corresponding lifting of $\lambda_0$ is 
    \[\lambda_{d, \rm O}=(\lambda_1,\lambda_3,\cdots,\lambda_{n-1};\lambda_2,\cdots,\lambda_n).\] 
    
\item If $n$ is odd, then the corresponding positive root system $\Psi_{b, \rm O}$ is generated by the simple roots
$\{ e_1-f_1,f_1-e_2,\cdots,f_{\frac{n-1}{2}}-e_{\frac{n+1}{2}}, e_{\frac{n+1}{2}}+f_{\frac{n-1}{2}} \}$. By the dominant condition, the corresponding lifting of $\lambda_0$ is \[\lambda_{d,\mathrm{O}}=(\lambda_1,\lambda_3,\cdots,\lambda_{n};\lambda_2,\cdots,\lambda_{n-1}).\] 
\end{enumerate}

In both cases, the generic discrete series representation corresponding to the basepoint $x_{b, {\rm O}}$ is determined by the Harish-Chandra parameter $(\lambda_{d,\rm O}, \Psi_{b, \rm O})$.
\end{enumerate}
\end{example}

Let $G$ be an equal rank real form  of ${\bf G}$.
In the following examples, we describe $\eta\in \widehat{A_{\varphi}}$ in terms of element of the Weyl group $W/W({\bf T}_*, {\bf K})$, where $\mathbf{K}$ is the complexification of the maximal compact subgroup of $G$.

\begin{example}\label{coordinate}
\begin{enumerate}
\item[(1)] Let $(\varphi, \eta)$ be a  discrete series Langlands-Vogan parameter of $G=\mathrm{Sp}_{2n}(\mathbb{R})$. Let $x_{\eta}\in \tilde{S}(z(\rho^{\vee}))$ be the strong involution corresponding to $\eta\in \widehat{A}_\varphi$. We follow the notations in \S \ref{HCSp}.
For any root $\alpha_i\in \Delta({\mathfrak g}, \mathfrak{t})$,
let $s_{\alpha_i}$ be the reflection associated to $\alpha_i$. In particular, we have 
$$s_{e_i+e_j}:\mathrm{diag}(\cdots,\underbrace{t_i}_{i\text{-th}},\cdots,\underbrace{-t_j}_{(2n-j+1)\text{-th}},\cdots)\mapsto\mathrm{diag}(\cdots,\underbrace{-t_j}_{i\text{-th}},\cdots,\underbrace{t_i}_{(2n-j+1)\text{-th}},\cdots).$$
The Weyl group $W$ of ${\bf G}=G\otimes {\bb C}$ is generated by the set 
\[\{s_{e_1-e_2},s_{e_2-e_3},\cdots,s_{e_{n-1}-e_n},s_{2e_n}\}\] of reflections. 

Let $A_{i,j}=\left(\begin{smallmatrix}
    I_{i-1}&&&&\\&0&&-1&\\&&I_{2n-j+1}&&\\&1&&0&\\&&&&I_{j-1}
\end{smallmatrix}\right)\in\mathrm{Norm}_{G_{\mathbb{C}}}(\mathbf{T}_*)$ for $1\leq i,j\leq n$. Then $A_{i,j}$ acts on $\mathbf{T}_*$ via \[ A_{i,j}:\mathrm{diag}(\cdots,\underbrace{r_i}_{i\text{-th}},\cdots,\underbrace{r_j^{-1}}_{(2n-j+1)\text{-th}},\cdots)\mapsto \mathrm{diag}(\cdots,\underbrace{r_j^{-1}}_{i\text{-th}},\cdots,\underbrace{r_i}_{(2n-j+1)\text{-th}},\cdots). \] Then the map sending $s_{e_i+e_j}$ to $A_{i,j}$ gives an explicit isomorphism $\mathrm{Norm}_{{\bf G}}(\mathbf{T}_*)/\mathbf{T}_*\cong W$. There exists an element $A\in\mathrm{Norm}_{{\bf G}}(\mathbf{T}_*)/\mathbf{T}_*$ such that $x_{\eta}=A\cdot x_{b,\mathrm{Sp}}$; thus such an $A$ gives an element  $w_\eta\in W$ via the above isomorphism and we get a Harish-Chandra parameter $(\lambda_\eta, \Psi_\eta,\mu_\eta)$, where $\lambda_\eta=w_\eta\cdot \lambda_{d,{\rm Sp}}$ and $\Psi_\eta=w_\eta\Psi_{*}$. 
\item[(2)] 
Let $(\varphi,\eta)$ be a discrete series Langlands-Vogan parameter of $G=\mathrm{O}(p,q)$ with $p\geq q$ even. Let $x_{\eta}\in \tilde{S}(z(\rho^{\vee}))$ be the strong involution corresponding to $\eta$. We use the notation in \S\ref{HCO}. 
For $\alpha\in\Delta(\mathfrak{g},\mathfrak{t})$, let $s_{\alpha}$ be the reflection associated to the root $\alpha$. Then the Weyl group $W$ of $G$ is generated by the set 
$$\{s_{e_i-e_{i+1}},s_{f_i- f_{j+1}}\}_{1\leq i< \frac{p}{2},1\leq j< \frac{q}{2}}\cup \{s_{e_i-f_i},s_{f_j-e_{j+1}},s_{f_{q_0}-e_{q_0+1}},\cdots,s_{f_{q_0}-e_{p_0}}\}_{1\leq i\leq \frac{p}{2},1\leq j< \frac{q}{2}}.$$ Moreover, the quotient group $W/W(\mathbf{T}_*,\mathbf{K})$ is generated by the image of the set of reflections $\{s_{e_i-f_i},s_{f_j-e_{j+1}},s_{f_{q_0}-e_{q_0+1}},\cdots,s_{f_{q_0}-e_{p_0}}\}_{1\leq i\leq \frac{p}{2},1\leq j< \frac{q}{2}}$.

The group theorectical description of the Weyl group $W$ is given by $W\cong\mathrm{Norm}_{G_{\bb C}}(\mathbf{T}_*)/\mathbf{T}_*$, where $\mathrm{Norm}_{G_{\bb C}}(\mathbf{T}_*)$ is the normalizer of ${\bf T}_*$ in $G_{\bb C}$. In particular, the following elements $S_{\pm i, \pm j}$ with $1\leq i\leq p_0$ and $1\leq j\leq q_0$ are in $\mathrm{Norm}_{G_{\bb C}}(\mathbf{T}_*)$. Let $J_2=\left(\begin{smallmatrix}
    1&0\\0&-1
\end{smallmatrix}\right)\in {\rm GL}_2({\bb C})$.  We set
$$ S_{i,j}=\left(\begin{smallmatrix}I_{2(i-1)}&&&&&\\&0&&&I_2&\\&&I_{2(p_0-i+1)}&&&\\&&&I_{2(j-1)}&&\\&I_2&&&0&\\&&&&&I_{2(q_0-j+1)}\end{smallmatrix}\right),$$
 $$S_{-i,j}=\left(\begin{smallmatrix}I_{2(i-1)}&&&&&\\&0&&&J_2&\\&&I_{2(p_0-i+1)}&&&\\&&&I_{2(j-1)}&&\\&I_2&&&0&\\&&&&&I_{2(q_0-j+1)}\end{smallmatrix}\right),$$
 $$ S_{i,-j}=\left(\begin{smallmatrix}I_{2(i-1)}&&&&&\\&0&&&I_2&\\&&I_{2(p_0-i+1)}&&&\\&&&I_{2(j-1)}&&\\&J_2&&&0&\\&&&&&I_{2(q_0-j+1)}\end{smallmatrix}\right),$$
$$ S_{-i,-j}=\left(\begin{smallmatrix}I_{2(i-1)}&&&&&\\&0&&&J_2&\\&&I_{2(p_0-i+1)}&&&\\&&&I_{2(j-1)}&&\\&J_2&&&0&\\&&&&&I_{2(q_0-j+1)}\end{smallmatrix}\right).$$

One can verify that the conjuate action of $S_{i,j}$ (resp. $S_{-i,j}$, $S_{i,-j}$ and $S_{-i,-j}$) on $\mathbf{T}_*$ coincides with the action of $s_{-e_i-f_j}$ (resp. $s_{e_i-f_j}$,$s_{-e_i+f_j}$ and $s_{e_i+f_j}$). Then this gives an explicit isomorphism $\mathrm{Norm}_{{\bf G}}(\mathbf{T}_*)/\mathbf{T}_*\cong W$. 
The similar argument as in the symplectic case gives an element $w_\eta\in W$ and a Harish-Chandra parameter $(\lambda_\eta, \Psi_\eta,\mu_\eta)$, where $\lambda_\eta=w_\eta\cdot \lambda_{d,{\rm O}}$ and $\Psi_\eta=w_\eta\Psi_{*}$. 
\end{enumerate}
\end{example}

If $\varphi$ is a limit of discrete series L-parameter. Following \cite[Remarque 5.4]{MD19}, we can modify the above discussion to relate the  Langlands-Vogan parameter $(\varphi, \eta)$ to its Harish-Chandra parameter. In fact, since our classical groups are of equal rank, we can realize the component group $A_{\varphi}$ as a quotient group of the component group of $A_{\varphi^{\rm reg}}$ with $\varphi^{\rm reg}$ a discrete series L-parameter. Thus, $\eta$ can be viewed as an element of $\widehat{A_{\varphi^{\rm reg}}}$ by composing with the quotient map $A_{\varphi^{\rm reg}}\rightarrow A_{\varphi}$. More precisely, if
$$\varphi=\bigoplus_{i=1}^kc_i\rho_{\lambda_i}\bigoplus (2z+1) {\bf 1},$$ where $z\in\mathbb{N}, \lambda_i\in 2\mathbb{N}$ and $\lambda_1>\cdots>\lambda_k>0$, then
the component group $A_{\varphi}$ is a quotient of \[A_{\varphi^{\rm reg}}=
\oplus_{i=1}^r ({\bb Z}/2{\bb Z})a_i \bigoplus \oplus_{j=1}^z({\bb Z}/2{\bb Z}) b_j.\] where $r=c_1+\cdots+c_k.$ Then, the character $\eta\in \widehat{A_{\varphi}}$ can be identified as an element $\eta_d=(\eta_1,\cdots,\eta_{r+z})\in\widehat{A_{\varphi^{\rm reg}}}$ with
\[\eta_1=\cdots=\eta_{c_1},\eta_{c_1+1}=\cdots=\eta_{c_1+c_2},\cdots,\eta_{\sum_{i=1}^{k-1}c_{i}+1}=\cdots=\eta_{r}, \eta_{r+1}=\cdots=\eta_{r+z}.\] in $\widehat{A_{\varphi^{\rm reg}}}$.

\section{Parameters of theta lift for symplectic-orthogonal dual pairs}\label{thetaPara}
In this section, we give a description of the explicit theta correspondence of equal rank groups via Langlands-Vogan parameters by combining the explicit theta correspondence of equal rank groups via Harish-Chandra parameters given by Moeglin (cf. \cite[\S 4]{Moe89}) and Paul (cf. \cite[Theorem~15]{Paul05}), and the explicit correspondence between the Langlands-Vogan parameters and the Harish-Chandra parameters.

Throughout this section, let $V$ be a $2n$-dimensional symplectic space over ${\bb R}$ and $V^{'}$ be a $(2n+2)$-dimensional orthogonal space over ${\bb R}$ with signature $(p, q)$. Then ${\rm Sp}(V)= {\rm Sp}_{2n}({\bb R})$ and ${\rm O}(V^{'})={\rm O}(p,q)$. For dual pair $({\rm Sp}_{2n}({\bb R}), {\rm O}(p,q))$, we will replace the notation $\theta_{V,V^{'}}$ by $\theta_{p,q}$.
\subsection{Theta liftings and Harish-Chandra parameters}

\begin{theorem}\cite[\S 4]{Moe89}\cite[Theorem 15]{Paul05}\label{correspondencetheorem}\footnote{In theorem \ref{Paul05JFA}, the situation is determined by the condition $e_{p_{0}+1}+e_{p_{0}+z}$ or $-(e_{p_{0}+1}+e_{p_{0}+z})$ occurs in $\Psi$. The difference between the  statement in \cite[Theorem 15]{Paul05} and our statement is due to the different choices of based root datum. In fact, Paul used the standard simple roots $\{e_i-e_{i+1},2e_{n},i=1,\cdots,n-1 \}$, and we use the non-compact simple roots $\{e_i+e_{n+1-i}, -e_{n+1-i}-e_{i+1} : 1\leq i \leq n-1 \}$.}\label{Paul05JFA}
 Let $\pi$ be a limit of discrete series representation of ${\rm Sp}_{2n}(\bb R)$ and $(\lambda_\pi,\Psi_\pi,\mu_\pi)$ be the Harish-Chandra parameter of $\pi$, where $$\lambda_\pi=(\underbrace{\lambda_{1},\cdots ,\lambda_{1}}_{p_{1}},\cdots,
\underbrace{\lambda_{k},\cdots,\lambda_{k}}_{p_{k}},\underbrace{0,\cdots,0}_{z},\underbrace{-\lambda_{k},\cdots,-\lambda_{k}}_{q_{k}},\cdots,\underbrace{-\lambda_{1},\cdots,-\lambda_{1}}_{q_{1}}).$$
Let $w=[\frac{z}{2}]$, $p_0=\sum_{i=1}^k p_i+w$ and $q_0=\sum_{i=1}^k q_i+w$. There is a unique pair of integers $(p,q)$ with $p+q=2n+2$ such that $\theta_{p,q}(\pi)$ is a non-zero limit of discrete series representation of ${\rm O}(p,q)$.
\begin{enumerate}
    \item $z=2w$: $\theta_{2p_0,2q_0}(\pi)\neq 0$ with the Harish-Chandra parameter $(\lambda_{0,0},1,\Psi_{0,0})$, where \begin{equation}
        \begin{split}
           \lambda_{0,0}=(&\underbrace{\lambda_{1},\cdots ,\lambda_{1}}_{p_{1}},\cdots,
\underbrace{\lambda_{k},\cdots,\lambda_{k}}_{p_{k}},\underbrace{0,\cdots,0}_{w},\\&\underbrace{\lambda_{1},\cdots,\lambda_{1}}_{q_{1}},\cdots,\underbrace{\lambda_{k},\cdots,\lambda_{k}}_{q_{k}},\underbrace{0,\cdots,0}_{w}),
        \end{split}
    \end{equation} and $\Psi_{0,0}$ is obtained from $\Psi_\pi$ as follows: for $1\leq i\leq p_0$ and $1\leq j\leq q_0$, the root $e_i-f_j\in\Psi_{0,0}$ if and only if $e_i+e_{n-j+1}\in\Psi$. (This determines $\Psi_{0,0}$ completely.)
\item $z=2w>0$: 
\begin{itemize}
\item If $e_{k+1}+e_{k+z}\in\Psi_\pi$, $\theta_{2p_0+2,2q_0}(\pi)\neq 0$ with the parameter $(\lambda_{2,0},1,\Psi_{2,0})$, where $\lambda_{2,0}$ is obtained from $\lambda_{0,0}$ by adding a zero on the left and $\Psi_{0,0}\subset \Psi_{2,0}$.
\item If $-e_{k+1}-e_{k+z}\in\Psi_\pi$, $\theta_{2p_0,2q_0+2}(\pi)\neq 0$ with the parameter $(\lambda_{0,2},1,\Psi_{0,2})$, where $\lambda_{0,2}$ is obtained from $\lambda_{0,0}$ by adding a zero on the right and $\Psi_{0,0}\subset \Psi_{0,2}$.
\end{itemize}
\item $z=w=0$: $\theta_{2p_0+2,2q_0}(\pi)\neq 0$ with parameter $(\lambda_{2,0},1,\Psi_{2,0})$ and $\theta_{2p_0,2q_0+2}(\pi)\neq 0$ with parameter $(\lambda_{0,2},1,\Psi_{0,2})$, where $\lambda_{2,0}$ and $\lambda_{0,2}$ are obtained from $\lambda_{0,0}$ by adding a zero on the left and right respectively, and $\Psi_{0,0}\subset \Psi_{2,0},\Psi_{0,2}$.
\item $z=2w+1$: 
\begin{itemize}
\item If $e_{k+1}+e_{k+z}\in\Psi_\pi$, then $\theta_{2p_0+2,2q_0+2}(\pi)\neq 0$ with the parameter $(\lambda_{1,1},1,\Psi_{1,1})$, where $\lambda_{1,1}$ is obtained from $\lambda_{0,0}$ by adding a zero on each side of the semicolon, and $\Psi_{0,0}\cup\{e_{p_0+1}-f_{q_0+1}\} \subset \Psi_{1,1}$. Moreover, $\theta_{2p_0+2,2q_0}(\pi)\neq 0$ with parameter $(\lambda_{1,0},1,\Psi_{1,0})$, where $\lambda_{1,0}$ is obtained from $\lambda_{0,0}$ by adding a zero on the left, and $\Psi_{0,0}\subset\Psi_{1,0}$. 
\item If $-e_{k+1}-e_{k+z}\in\Psi_\pi$, then $\theta_{2p_0+2,2q_0+2}(\pi)\neq 0$ with the parameter $(\lambda_{1,1},1,\Psi_{1,1})$, where $\lambda_{1,1}$ is obtained from $\lambda_{0,0}$ by adding a zero on each side of the semicolon, and $\Psi_{0,0}\cup\{-e_{p_0+1}+f_{q_0+1}\} \subset \Psi_{1,1}$. $\theta_{2p_0+2,2q_0}(\pi)\neq 0$ with parameter $(\lambda_{0,1},1,\Psi_{0,1})$, where $\lambda_{0,1}$ is obtained from $\lambda_{0,0}$ by adding a zero on the right, and $\Psi_{0,0}\subset\Psi_{0,1}$. 
\end{itemize}
\end{enumerate}

\end{theorem}
\begin{remark}
    In the Paul's theorem \cite[Theorem 15]{Paul05}, there are exactly four pairs of integers $(p,q)$ with $p+q=2n$ or $2n+2$ such that $\theta_{p,q}(\pi)\neq 0$. But there is only one pair of integers $(p,q)$ with $p+q=2n+2$ such that $\theta_{p,q}(\pi)$ is a non-zero limit of discrete series representation of $\mathrm{O}(p,q)$.
\end{remark}

\subsection{Translation}

Let $\varphi:W_{\mathbb{R}}\rightarrow {\rm O}(M)$ be a L-parameter of ${\rm Sp}(V)$, where $M$ is a $(2n+1)$-dimensional orthogonal space. Let $\mu$ be the infinitesimal character associated to $\varphi$ and $\Pi(\varphi)$ the L-packet associated to $\varphi$. Let $\pi\in \Pi(\varphi)$ be a limit of discrete series representation of ${\rm Sp}(V)$ with Langlands-Vogan parameter $(\varphi, \eta)$. Suppose $\varphi$ admits a decomposition as in Example \ref{Lparameterdecom} (1). Let $\lambda_{d,\mathrm{Sp}}$ be the Harish-Chandra parameter of generic discrete series representation of $\mathrm{Sp}_{2n}(\mathbb{R})$ as in Example \ref{BaseSp}.

\begin{proposition}
    Let $\pi$ be a generic discrete series representation of $\mathrm{Sp}(V)$ corresponding to the basepoint $x_{b,\mathrm{Sp}}$. Then $\theta_{V,V^{'}}(\pi)$ is a generic discrete series representation of $\mathrm{O}(V^{'})$ with signature $(p,q)=\begin{cases}(n+2,n), & \text{ if } n \text{ is even}\\
(n+1,n+1), &  \text{ if } n \text{ is odd}\end{cases}$, which corresponds to the basepoint $x_{b,\mathrm{O}}$.
\end{proposition}
\begin{proof}
    This proposition follows from our description of the Harish-Chandra parameters of the basepoints $x_{b,\mathrm{Sp}}$ and $x_{b,\mathrm{O}}$ in Example \ref{BaseSp}, and the explicit theta correspondence in Theorem \ref{correspondencetheorem}. In the following, we provide the explicit computation according to the parity of $n$.
\begin{enumerate}[(1)]

\item Assume $n$ is even. The Harish-Chandra parameter $\lambda_{d,\mathrm{Sp}}$ of $\pi$ has the form \[ (\lambda_1,\lambda_3,\cdots,\lambda_{n-1},-\lambda_n,\cdots,-\lambda_2), \] and the corresponding root system is generated by the simple roots \[\{ e_1+e_n,-e_n-e_2,\cdots,e_{\frac{n}{2}}+e_{\frac{n}{2}+1},-2e_{\frac{n}{2}+1} \}.\] Hence by case $3$ of theorem \ref{Paul05JFA}, the Harish-Chandra parameter of $\theta_{V,V^{'}}(\pi)$ is the pair $(\lambda_{2,0},1, \Psi_{2,0})$, where \[\lambda_{2,0}=(\lambda_1,\lambda_3,\cdots,\lambda_{n-1},0;\lambda_2,\cdots,\lambda_n)\] and 
 the corresponding root system   \[ \Psi_{2,0}=\langle e_1-f_1,f_1-e_2,\cdots,e_{\frac{n}{2}-1}-f_{\frac{n}{2}},f_{\frac{n}{2}}-e_{\frac{n}{2}},f_{\frac{n}{2}}+e_{\frac{n}{2}}\rangle.\]Since $\lambda_1>\cdots>\lambda_n>0$ and the root system is generated by the non-compact simple roots, the parameter $(\lambda_{2,0},1,\Psi_{2,0})$ is exactly the Harish-Chandra parameter of the generic discrete series representation of $\mathrm{O}(n+2,n)$ as in Example \ref{BaseSp}. As a result, $\theta_{V,V^{'}}(\pi)$ is a generic discrete series of $\mathrm{O}(n+2,n)$.  
    
\item Assume $n$ is odd. The Harish-Chandra parameter $\lambda_{d,\mathrm{Sp}}$ of $\pi$ has the form \[ (\lambda_1,\lambda_3,\cdots,\lambda_{n},-\lambda_{n-1},\cdots,-\lambda_2), \] and the corresponding root system is generated by the simple roots \[\Psi_{\mathrm{Sp}}=\langle e_1+e_n,-e_n-e_2,\cdots,-e_{\frac{n+1}{2}}-e_\frac{n+3}{2},2e_{\frac{n+1}{2}} \rangle.\] 

Hence by case $3$ of Theorem \ref{Paul05JFA}, the Harish-Chandra parameter of $\theta_{V,V^{'}}(\pi)$ is $(\lambda_{0,2},1,\Psi_{0,2})$ where \[\lambda_{0,2}=(\lambda_1,\lambda_3,\cdots,\lambda_{n};\lambda_2,\cdots,\lambda_{n-1},0).\] and the corresponding root system \[\Psi_{0,2}=\langle e_1-f_1,f_1-e_2,\cdots,f_{\frac{n-1}{2}}-e_{\frac{n-1}{2}}, e_{\frac{n-1}{2}}-f_{\frac{n+1}{2}},f_{\frac{n+1}{2}}-e_{\frac{n+1}{2}} ,f_{\frac{n+1}{2}}+e_{\frac{n+1}{2}}\rangle.\] Similarly, the parameter $(\lambda_{0,2},\Psi_{0,2})$ is exactly the Harish-Chandra parameter of the generic discrete series representation of $\mathrm{O}(n+1,n+1)$ as in Example \ref{BaseSp}. Hence $\theta_{V,V^{'}}(\pi)$ is a generic discrete series of ${\rm O}(n+1,n+1)$.

\end{enumerate}
\end{proof}

Now, suppose  $\pi$ is a limit of discrete series representation of ${\rm Sp}(V)$. Then the Langlands-Vogan parameter $(\varphi, \eta)$ of $\pi$ determines a Harish-Chandra parameter $\lambda_\eta$ is of the form 
\[(\underbrace{\lambda_{1},\cdots ,\lambda_{1}}_{p_{\eta,1}},\cdots,
\underbrace{\lambda_{k},\cdots,\lambda_{k}}_{p_{\eta,k}},\underbrace{0,\cdots,0}_{z},\underbrace{-\lambda_{k},\cdots,-\lambda_{k}}_{q_{\eta,k}},\cdots,\underbrace{-\lambda_{1},\cdots,-\lambda_{1}}_{q_{\eta,1}}), \]
and  a positive root system $\Psi_{\eta}$. 

Let $p_{\eta,0}=\sum\limits_{l=1}^kp_{\eta,l}$ and $q_{\eta,0}=\sum\limits_{l=1}^{k}q_{\eta,l}$, and $w=[\frac{z}{2}]$.
We set $p_\eta=p_{\eta,0}+w$ and $q_{\eta}=q_{\eta,0}+w$. Note that \[p_\eta+q_\eta=\begin{cases}n, & \text{ if } z\equiv 0\mod 2;\\ n-1, & \text{ if } z\equiv 1\mod 2.\end{cases}\]
By Thereom \ref{Paul05JFA}, we set 
\[\lambda_{\eta,0,0}=(\underbrace{\lambda_{1},\cdots ,\lambda_{1}}_{p_{\eta,1}},\cdots,
\underbrace{\lambda_{k},\cdots,\lambda_{k}}_{p_{\eta,k}},\underbrace{0,\cdots,0}_{w};\underbrace{\lambda_{1},\cdots,\lambda_{1}}_{q_{\eta,1}},\cdots,\underbrace{\lambda_{k},\cdots,\lambda_{k}}_{q_{\eta,k}},\underbrace{0,\cdots,0}_{w}), \]
and a root system $\Psi_{\eta,0,0}$ obtained from $\Psi_{\eta}$ as follows: for $1\leq i\leq p_\eta$ and $1\leq j\leq q_\eta$, the root $e_i-f_j\in \Psi_{\eta,0,0}$ if and only if $e_i+e_{n-j+1}\in \Psi_\eta$. According to the values of $z$ and $w$, we have: 
\begin{enumerate}
\item[(1)] If $z=w=0$ (i.e. case $3$ of Theorem \ref{Paul05JFA}), then we have $p=2p_{\eta}+2$ and $q=2q_{\eta}$. The corresponding Harish-Chandra parameter is 
$(\lambda_{\eta, 2,0},1,\Psi_{\eta,2,0})$,
where $\lambda_{\eta, 2,0}$ is obtained from $\lambda_{\eta,0,0}$ by adding a zero on the left and $\Psi_{\eta, 0,0}\subset \Psi_{\eta,2,0}$. 
\item[(2)] If $z=2w>0$ (i.e. case $2$ of Theorem \ref{Paul05JFA}), there are two possible cases: 
\begin{enumerate}
    \item If $e_{p_{\eta,0}+1}+e_{p_{\eta,0}+z}\in\Psi_\eta$, then we have $p=2p_\eta+2$ and $q=2q_{\eta}$. The corresponding Harish-Chandra parameter is 
$(\lambda_{\eta,2,0},1,\Psi_{\eta, 2,0})$,
where $\lambda_{\eta,2,0}$ is obtained from $\lambda_{\eta,0,0}$ by adding a zero on the left side, and $\Psi_{\eta, 2,0}$ contains $\Psi_{\eta,0,0}$.
\item If $-e_{p_{\eta,0}+1}-e_{p_{\eta,0}+z}\in\Psi_{\eta}$, then we have $p=2p_\eta$ and $q=2q_{\eta}+2$. The corresponding Harish-Chandra parameter is 
$(\lambda_{\eta,0,2},1,\Psi_{\eta, 0,2})$,
where $\lambda_{\eta,0,2}$ is obtained from $\lambda_{\eta,0,2}$ by adding a zero on the right side, and $\Psi_{\eta, 0,2}$ contains $ \Psi_{\eta,0,0}$. 
\end{enumerate} 
\item[(3)] If $z=2w+1$ (i.e. case $4$ of Theorem \ref{Paul05JFA}), then we have $p=2p_{\eta}+2$ and $q=2q_{\eta}+2$. The corresponding Harish-Chandra parameter is 
$(\lambda_{\eta,1,1},1,\Psi_{\eta, 1,1})$,
where $\lambda_{\eta,1,1}$ is obtained from $\lambda_{\eta,0,0}$ by adding a zero on each side of the semicolon. Moreover, 
\begin{enumerate}
\item if $e_{p_{\eta,0}+1}+e_{p_{\eta,0}+z}\in\Psi_\eta$, then $e_{p_{\eta,0}+w+1}-f_{q_{\eta,0}+w+1}\in\Psi_{\eta,1,1}$. 
\item If $-e_{p_{\eta,0}+1}-e_{p_{\eta,0}+z}\in\Psi_\eta$, then $-e_{p_{\eta,0}+w+1}+f_{q_{\eta,0}+w+1}\in\Psi_{\eta,1,1}$.
\end{enumerate}
\end{enumerate}

At last, we need to translate this Harish-Chandra parameter $(\lambda_{\eta,a,b},1,\Psi_{\eta,a,b})$ of $\theta_{V,V^{'}}(\pi)$ into the Langlands-Vogan parameter of $\theta_{V,V^{'}}(\pi)$, where $(a,b)=(2,0)$, $(0,2)$ or $(1,1)$, which depends on the Langlands-Vogan parameters $(\lambda_{\eta},\Psi_{\eta})$ of $\pi$.
\begin{enumerate}
    \item[(1)] If $(a,b)=(2,0)$, then the corresponding Harish-Chandra parameter is \[\lambda_{\eta,2,0}=(\underbrace{\lambda_{1},\cdots ,\lambda_{1}}_{p_{\eta,1}},\cdots,
\underbrace{\lambda_{k},\cdots,\lambda_{k}}_{p_{\eta,k}},\underbrace{0,\cdots,0}_{w+1};\underbrace{\lambda_{1},\cdots,\lambda_{1}}_{q_{\eta,1}},\cdots,\underbrace{\lambda_{k},\cdots,\lambda_{k}}_{q_{\eta,k}},\underbrace{0,\cdots,0}_{w}) \]
and a root system $\Psi_{\eta,2,0}\supset\Psi_{\eta,2,0}$.
\item[(2)] If $(a,b)=(0,2)$, then the corresponding Harish-Chandra parameter is \[\lambda_{\eta,2,0}=(\underbrace{\lambda_{1},\cdots ,\lambda_{1}}_{p_{\eta,1}},\cdots,
\underbrace{\lambda_{k},\cdots,\lambda_{k}}_{p_{\eta,k}},\underbrace{0,\cdots,0}_{w};\underbrace{\lambda_{1},\cdots,\lambda_{1}}_{q_{\eta,1}},\cdots,\underbrace{\lambda_{k},\cdots,\lambda_{k}}_{q_{\eta,k}},\underbrace{0,\cdots,0}_{w+1}) \]
and a root system $\Psi_{\eta,0,2}\supset\Psi_{\eta,0,0}$.
\item[(3)] If $(a,b)=(1,1)$, then the corresponding Harish-Chandra parameter is \[\lambda_{\eta,2,0}=(\underbrace{\lambda_{1},\cdots ,\lambda_{1}}_{p_{\eta,1}},\cdots,
\underbrace{\lambda_{k},\cdots,\lambda_{k}}_{p_{\eta,k}},\underbrace{0,\cdots,0}_{w+1};\underbrace{\lambda_{1},\cdots,\lambda_{1}}_{q_{\eta,1}},\cdots,\underbrace{\lambda_{k},\cdots,\lambda_{k}}_{q_{\eta,k}},\underbrace{0,\cdots,0}_{w+1}) \]
and a root system $\Psi_{\eta,1,1}\supset\Psi_{\eta,0,0}$.
\end{enumerate}
Hence the corresponding Langlands-Vogan parameter of $\theta_{V,V^{'}}(\pi)$ is the pair $(\theta_{V,V^{'}}(\varphi),\theta_{V,V^{'}}(\eta))$ given by \[\theta_{V,V^{'}}(\varphi)=\oplus_{i=1}^k(p_{\eta,i}+q_{\eta,i})\rho_{\lambda_i}\bigoplus(2w+2)\mathbf{1} \] and \[\theta_{V,V^{'}}(\eta)\vert_{A_{\varphi}}=\eta.\]

Conversely, we start with a limit of discrete series representation $\sigma$ of $\mathrm{SO}(p,q)$ with $p+q=2n+2$. Moreover, we assume the theta lift of $\sigma$ (as an ${\rm O}(p,q)$ representation) is non-zero.  

Let $(\varphi^{'},\eta^{'})$ be the Langlands-Vogan parameter of $\sigma$. By \cite[Corollary 24]{Paul05}, the Harish-Chandra lifting of $\sigma$ has the form
\[\lambda^{'}_{\eta^{'}}=(\underbrace{\lambda^{'}_{1},\cdots ,\lambda^{'}_{1}}_{p_{\eta^{'},1}},\cdots,
\underbrace{\lambda^{'}_{k},\cdots,\lambda^{'}_{k}}_{p_{\eta^{'},k}},\underbrace{0,\cdots,0}_{z+1};\underbrace{\lambda^{'}_{1},\cdots,\lambda^{'}_{1}}_{q_{\eta^{'},1}},\cdots,\underbrace{\lambda^{'}_{k},\cdots,\lambda^{'}_{k}}_{q_{\eta^{'},k}},\underbrace{0,\cdots,0}_{z^{'}}),
\]
where $\lambda^{'}_1>\cdots>\lambda^{'}_k>0$, $2(p_{\eta^{'},1}+\cdots+p_{\eta^{'},k}+z+1)=p$ and $2(q_{\eta^{'},1}+\cdots+q_{\eta^{'},k}+z^{'})=q$. This implies $\varphi^{'}$ has the decomposition \[\varphi^{'}=\oplus_{i=1}^k(p_{\eta,i}+q_{\eta,i})\rho_{\lambda_i}\bigoplus 2(z+z^{'}+1)\mathbf{1}.\]

Denote by $\theta(\sigma)$ the corresponding representation of $\mathrm{Sp}_{2n}(\mathbb{R})$. The theta correspondence gives the Harish-Chandra lifting of $\theta(\sigma)$ is \[\lambda_{\eta^{'}}=(\underbrace{\lambda^{'}_{1},\cdots ,\lambda^{'}_{1}}_{p_{\eta^{'},1}},\cdots,
\underbrace{\lambda^{'}_{k},\cdots,\lambda^{'}_{k}}_{p_{\eta^{'},k}},\underbrace{0,\cdots,0}_{z+z^{'}};\underbrace{-\lambda^{'}_{k},\cdots,\lambda^{'}_{k}}_{q_{\eta^{'},k}},\cdots,\underbrace{-\lambda^{'}_{1},\cdots,-\lambda^{'}_{1}}_{q_{\eta^{'},1}})\] and the Langlands parameter $\varphi$ of $\theta(\sigma)$ is given by \[\varphi=\oplus_{i=1}^k(p_{\eta,i}+q_{\eta,i})\rho_{\lambda_i}\bigoplus(2(z+z^{'})+1)\mathbf{1}.\] The component group $A_\varphi$ is a subgroup of $A_{\varphi^{'}}$. Moreover, the Vogan parameter is $$\theta_{V,V^{'}}(\eta)=\eta\vert_{A_{\varphi}}.$$

We summary our description of theta correspondence for symplectic-orthogonal dual pairs in term of Langlands-Vogan parameters for limit of discrete series representations in the following theorem. 

\begin{theorem}\label{LVHCtheorem}
  \begin{enumerate}
  \item[(1)]   If $\pi$ is a limit of discrete series representation of $\mathrm{Sp}(V)$ with Langlands-Vogan parameter $(\varphi,\eta)$, where \[\varphi=\oplus_{i=1}^k(p_{\eta,i}+q_{\eta,i})\rho_{\lambda_i}\bigoplus (2z+1)\mathbf{1},\] with $p_{\eta,i},q_{\eta,i},z\in\mathbb{N}$, $i=1,\cdots,k$, $\sum_{i=1}^k (p_{\eta,i}+q_{\eta,i})+z=n$ and $\rho_{\lambda_i}$ is self-dual irreducible representation of the Weil group $W_{\mathbb{R}}$ of dimension 2.
     Then there exists a unique pair of even integers $(p,q)$ such that $p+q=2n+2$, and $\theta_{V,V^{'}}(\pi)$ is a limit of discrete series representation of $\mathrm{O}(p,q)$ with Langlands parameter $\theta_{V,V^{'}}(\varphi)$, where \[\theta_{V,V^{'}}(\varphi)=\oplus_{i=1}^k(p_{\eta,i}+q_{\eta,i})\rho_{\lambda_i}\bigoplus(2z+2)\mathbf{1}.\] Moreover, we can regard the component group $A_{\varphi}$ as a subgroup of $A_{\theta_{V,V^{'}}(\varphi)}$, then \[ \theta_{V,V^{'}}(\eta)\vert_{A_{\varphi}}=\eta. \] 
     \item [(2)] Let $V^{'}$ be a $(2n+2)$-dimensional real orthogonal space with signature $(p,q)$, where $p,q$ are even integers, and let $\pi^{'}$ be a limit of discrete series representation of $\mathrm{O}(V^{'})$ with Langlands-Vogan parameter $(\varphi^{'},\eta^{'})$. Assume that $\theta_{V^{'},V}(\pi^{'})\neq 0$. Then the Langlands-Vogan parameter $(\varphi,\eta)$ of the representation $\theta_{V^{'},V}(\pi^{'})$ of $\mathrm{Sp}(V)$ is given by
\[\varphi=\varphi^{'}-\mathbf{1}, \quad \eta=\eta^{'}\vert_{A_{\varphi}}.\]
     \end{enumerate}
\end{theorem}

\subsection{The Tempered Case}
In this section, we recall the Langlands-Vogan parametrization of parabolic inductions (cf. Proposition \ref{bijectiontolds}) and the induction principle of theta lifts (cf. Theorem \ref{Inductionprinc'}), which allow us to reduce the proof of our main theorem for the tempered case (cf. Theorem \ref{Maintemper}) to the case of limit of discrete series (i.e. Theorem \ref{correspondencetheorem}).

\subsubsection{Langlands parametrization of parabolic inductions}
Let $H$ be a real symplectic group or a real orthogonal group. 
Let $H_0$ be the subgroup of $H$ which is the same type of $H$, i.e. $H_0=\mathrm{Sp}(V_0)$ or $\mathrm{O}(V_0^{'})$ for some symplectic subspace $V_0\subset V$ or some orthogonal subspace $V_0^{'}\subset V^{'}$. Consider the parabolic subgroup $P=MAN$ of $H$ with Levi factor \[ MA\cong H_0 \times\mathrm{GL}_2(\bb R)^s\times\mathrm{GL}_1(\bb R)^t \]
where $s,t$ are non-negative integers. The parabolically induced representations 
\[  { \rm Ind }_P^{G}\pi_0\otimes\tau\otimes\chi\otimes 1, \]
have a unique irreducible Langlands quotient, where $\pi_0$ is a limit of discrete series of $H_0$, $\tau$ is a relative limit of discrete series representation of $\mathrm{GL}_2(\bb R)^s$ and $\chi$ is a quasi-character of $\mathrm{GL}_{1}(\bb R)^t$. Denote this unique irreducible Langlands quotient representation by $\pi=\pi(\pi_0,\tau,\chi)$.

Let $\varphi$ be the Langlands parameter associated to $\pi$ with the decomposition \[\varphi=\oplus_{i=1}^kc_i\rho_{\lambda_i}\bigoplus (2z+1)\mathbf{1}\bigoplus C,\] where 
\begin{enumerate}
\item $\lambda_i$ are odd positive integers such that $\lambda_1>\cdots>\lambda_k>0$;
\item $\rho_{\lambda_i}$'s are self-dual irreducible representations of the Weil group $W_{\mathbb{R}}$ of dimension $2$;
\item  the $c_i, z$ are natural numbers for $i=1,\cdots,k$ such that $z + \sum_{i=1}^k c_i=n$;
\item $C$ is the non self-dual part of $\varphi$.
\end{enumerate}

Let $\varphi_0=\oplus_{i=1}^kc_i\rho_{\lambda_i}\bigoplus (2z+1)\mathbf{1}$ which is a Langlands parameter of $H_0$. Then $\Pi(\varphi,H)$ is the collection of all these Langlands quotients where ${\pi}_0$ ranges over $\Pi(\varphi_0,H_0)$.

\begin{proposition}\label{bijectiontolds}\cite[Theorem~2.9]{Vogan84}
    Taking Langlands quotients of parabolic inductions gives a bijection between $\Pi(\varphi,H)$ and $\Pi(\varphi_0,H_0)$.
\end{proposition}
Recall that by \cite[Theorem~6.3]{Vogan93}, we have a natural bijection between $\Pi(\varphi,H)$ (resp. $\Pi({\varphi_0}, H_0)$) and the set of irreducible representations of the component group $A_\varphi:=\pi_0(C_\varphi)$ (resp. $A_{\varphi_0}$) of $\varphi$ (resp. $\varphi_0$). The following diagram commutes
\[\begin{tikzcd}
	{\Pi({\varphi_0})} & {\widehat{A_{\varphi_0}}} \\
	{\Pi(\varphi)} & {\widehat{A_{\varphi}}}
	\arrow[from=1-1, to=1-2]
	\arrow["1:1", from=1-1, to=2-1]
	\arrow["{=}", from=1-2, to=2-2]
	\arrow[from=2-1, to=2-2]
\end{tikzcd}\]
where the bijection on the left arrow is given by Proposition \ref{bijectiontolds}.
\subsubsection{Induction principle of theta lift}\label{Inducprinsec}
Let $(G,G^{'})=(\mathrm{Sp}_{2n}(\bb R), \mathrm{O}(p,q))$ be the symplectic-orthogonal dual pair with $p+q=2n+2$ and $\omega$ the oscillator representation for the dual pair $(G, G')$. There are parabolic subgroups $P=MAN$ and $P^{'}=M^{'}A^{'}N^{'}$ of $G$ and $G^{'}$ with Levi factor \[MA\cong \mathrm{Sp}_{2n^{'}}(\bb R)\times \mathrm{GL}_2(\bb R)^s\times\mathrm{GL}_1(\bb R)^t \] and \[M'A'\cong \mathrm{O}(p^{'},q^{'})\times \mathrm{GL}_2(\bb R)^s\times\mathrm{GL}_1(\bb R)^t  \]
where $2n^{'}+2s+t=2n$ and $p^{'}+q^{'}=2n^{'}+2$.

Recall that we denote by $\chi_{\epsilon, \kappa}$ the quasi-character $x\mapsto \mathrm{sgn}(x)^{\frac{\epsilon-1}{2}}|x|^{\kappa}$ of $\mathrm{GL}_1(\bb R)$ for $\epsilon \in \{\pm 1\}$ and $\kappa\in {\bb C} $. Let $\pi=\pi(\rho,\tau,\chi_{\epsilon,\kappa})$ be an irreducible admissible representation of ${\rm Sp}_{2n}({\bb R})$, which is the unique irreducible quotient of the standard module of $\rho \otimes \tau\otimes \chi$ with $\rho$ a limit of discrete series representation of ${\rm Sp}_{2n^{'}}({\bb R})$, $\tau=\otimes_{i=1}^s\tau(\mu_i,\nu_i)$ a relative limit of discrete series of ${\rm GL}_2(\bb R)^s$ and $\chi=\otimes_{i=1}^t \chi_{\epsilon_i, \kappa_i}$ a character of ${\rm GL}_1(\bb R)^t$ . As in Theorem \ref{correspondencetheorem}, there is a unique pair of integers $(p^{'},q^{'})$ with $p^{'}+q^{'}=2n^{'}+2$ such that $\theta_{p^{'},q^{'}}(\rho)$ is a non-zero limit of discrete series representation of ${\rm O}(p^{'},q^{'})$. 
\begin{theorem}\cite[Theorem 18]{Paul05}\label{Inductionprinc'}
Let $n,n^{'},p^{'},q^{'},s$ and $t$ be non-negative integers as above. Let $\epsilon_{p,q}=(\epsilon_1\cdot(-1)^{\frac{p^{'}-q^{'}}{2}},...,\epsilon_t\cdot(-1)^{\frac{p^{'}-q^{'}}{2}}) $. For  the irreducible admissible representation $\pi=\pi(\rho,\tau,\chi_{\epsilon,\kappa})$ of ${\rm Sp}_{2n}({\bb R})$ as above, we have
$$\theta_{p,q}(\pi)=\theta_{p,q}(\pi)(\theta_{p',q'}(\rho),\tau,\chi_{\epsilon_{p,q},\kappa}).$$
\end{theorem}
\begin{remark} We explain briefly the proof of above theorem to see that our choice of root datum doesn't affect the result. Let $\pi^{'}$ be the proposed theta lift of $\pi$ and $\rho^{'}=\theta_{p',q'}(\rho)$. We can arrange that $\pi$ is a quotient of $I=\mathrm{Ind}_{P}^{G}(\rho\otimes \tau\otimes\chi_{\epsilon,\kappa}\otimes\mathbf{1})$. Then $\pi'$ is a constituent of $I'=\mathrm{Ind}_{P'}^{G'}(\rho'\otimes \tau^*\otimes\chi^*_{\epsilon,\kappa}\otimes\mathbf{1})$, where $*$ is the contragredient representation. By the induction principle, there is a non-zero $G\times G'$-equivariant map $\omega\rightarrow I\times I'$. Suppose $\pi$ has a minimal $K$-type $\Lambda$. Then $\pi$  occurs in the correspondence and lifts to a constituent of $I'$ containing the $K$-type $\Lambda'$ which corresponds to $\Lambda$. This is minimal $K$-type of $\pi'$, and hence a minimal $K$-type of $I'$. Since the minimal $K$-types of $I'$ have multiplicity one, $\theta_{p,q}(\pi)=\pi'$. As a result, our choice of root datum doesn't affect the result. 
\end{remark}

By Proposition \ref{bijectiontolds} and the identification $\widehat{A_\varphi}=\widehat{A_{\varphi_0}}$, this induction principle implies our main result in the tempered case.
\begin{theorem}\label{Maintemper}
Let $V$ be a $2n$-dimensional symplectic space over $\mathbb{R}$.
    
$(1)$ Let $\pi$ be a tempered representation of $\mathrm{Sp}(V)$ with Langlands-Vogan parameter $(\varphi,\eta)$. Then there exists a unique pair of even integers $(p,q)$ satisfying $p+q=2n+2$ and a $(2n+2)$-dimensional orthogonal space $V^{'}$ with signature $(p,q)$ such that $\theta_{p,q}(\pi)$ is a tempered representation of $\mathrm{O}(p,q)$ with Langlands-Vogan parameter ($\theta_{p,q}(\varphi)$, $\theta_{p,q}(\eta))$, where $\theta_{p,q}(\varphi)=\varphi+{\bf 1}$. Then the component group $A_{\varphi}$ is a subgroup of $A_{\theta_{p,q}(\varphi)}$, and we have
$$\theta_{p,q}(\eta)|_{A_\varphi}=\eta.$$

$(2)$ Let $V^{'}$ be a $(2n+2)$-dimensional real orthogonal space with signature $(p,q)$, where $p,q$ are even integers, and let $\pi^{'}$ be a tempered representation of $\mathrm{O}(V^{'})$ with Langlands-Vogan parameter $(\varphi^{'},\eta^{'})$. Assume that $\theta_{V^{'},V}(\pi^{'})\neq 0$. Then the Langlands-Vogan parameter $(\varphi,\eta)$ of the tempered representation $\theta_{V^{'},V}(\pi^{'})$ of $\mathrm{Sp}(V)$ is given by
\[\varphi=\varphi^{'}-\mathbf{1}, \quad \eta=\eta^{'}\vert_{A_{\varphi}}.\]
\end{theorem}


\end{document}